\documentclass[a4paper,10pt]{article}

\usepackage{amsmath}  
\usepackage{amssymb}            
\usepackage{amsxtra}
\usepackage{amscd}
\usepackage[mathscr]{euscript}

\everymath{\displaystyle} 

\newtheorem{thm}{Theorem}[subsection]
\newtheorem{prop}[thm]{Proposition}

\newtheorem{rem}[thm]{Remark}
\newtheorem{conj}[thm]{Conjecture}

\newtheorem{assu}[thm]{Assumption}

\newcommand{\npmod}[1]{\!\!\pmod{#1}}
\newcommand{\nnpmod}[1]{\!\!\!\!\pmod{#1}}

\newenvironment{proof}{\par\noindent{\bf[Proof]}}%
                      {$\blacksquare$\noindent\par\vspace{0.5\baselineskip}}
                      {$\blacksquare$\par\noindent}

\font\b=cmr10 scaled \magstep4

\def\bigzerou{\smash{\lower0.7ex\hbox{\b 0}}}
\def\bigastl{\smash{\lower0.7ex\hbox{\b *}}}
\def\bigastu{\smash{\lower2.7ex\hbox{\b *}}}

\def\addots{\mathinner
    {\mkern1mu\raise1pt\hbox{.}\mkern2mu
    \raise4pt\hbox{.}\mkern2mu\raise7pt\vbox{\kern7pt\hbox{.}}\mkern1mu}}

\makeatletter
\renewcommand{\subsection}{\@startsection%
  {subsection}%
  {2}%
  {0mm}%
  {\baselineskip}%
  {-0.2\parindent}%
  {\normalfont\normalsize\upshape\bfseries}}%
\def\@dotsep{1.5}
\def\@pnumwidth{1em}
\makeatother

\title{Regular irreducible characters \\of a hyperspecial compact group} 
\author{Koichi Takase
        \thanks{The author is partially supported by 
                    JSPS KAKENHI Grant Number JP 16K05053}}
\date{}

\begin{document}   


%
%

\maketitle

\begin{abstract}
A parametrization of irreducible unitary representations 
associated with the regular adjoint orbits of a
hyperspecial compact subgroup of a reductive group over a non-dyadic
non-archimedean local filed 
is presented. The parametrization is given by means of (a subset of) 
the character group of certain finite abelian groups arising from the
reductive group. Our method 
is based upon Cliffod's theory and Weil representations over  
finite fields. It works under an assumption of 
the triviality of certain Schur multipliers defined for an algebraic
group over a finite field. The assumption of 
the triviality has good evidences in the case of
general linear groups and highly probable in general. 

\end{abstract}

\section{Introduction}\label{sec:introduction}

Let $F$ be a non-dyadic non-archimedean local field. 
The integer ring of $F$ is denoted by $O$ with the maximal
ideal $\frak{p}$ generated by $\varpi$. The residue class field 
$\Bbb F=O/\frak{p}$ is a finite field of $q$ elements.
Fix a continuous unitary character $\tau$ of the additive group $F$ such
that $\{x\in F\mid\tau(xO)=1\}=O$, and define an additive
character $\widehat\tau$ of $\Bbb F$ by 
$\widehat\tau(\overline x)=\tau(\varpi^{-1}x)$. 
For an integer $l>0$ put $O_l=O/\frak{p}^l$ so that $\Bbb F=O_1$. 

If a connected reductive quasi-split linear group $\bold{G}$ over
$F$ is split over an
unramified extension of $F$, then there exists a smooth affine group
scheme $G$ over $O$ such that $G{\otimes}_{O}F=\bold{G}$ and 
$G{\otimes}_{O}\Bbb F$ is a connected reductive group over $\Bbb F$. 
In this case the locally compact group $G(F)=\bold{G}(F)$ of the
$F$-rational points has an open compact subgroup $G(O)$ wich is
called a hyperspecial compact subgroup of $G(F)$ 
\cite[3.8.1]{Tits1979}. 
An important problem in the harmonic analysis on $G(F)$
is to determine the irreducible unitary representations of the compact
group $G(O)$. Such a representation $\pi$ of $G(O)$ factors through  the 
canonical group homomorphism $G(O)\to G(O_r)$ for some
$r>0$ since the canonical group homomorphism is surjective due to 
the smoothness of the group scheme $G$ over $O$ and Hensel's
lemma, and $\pi$ is trivial on the kernel of the canonical group
homomorphism for some $r>0$. 
Hence the problem is reduced to determine the set 
$\text{\rm Irr}(G(O_r))$ of the equivalence classes of the 
irreducible complex representations of the finite group $G(O_r)$. 

This problem in the case $r=1$, that is the representation theory of
the finite reductive group $G(\Bbb F)$, has been studied extensively,
starting 
from Green \cite{Green1955} concerned with $GL_n(\Bbb F)$ to the
decisive paper of Deligne-Lusztig \cite{Deligne-Lusztig1976}. 

This paper treats the case $r>1$ where the study of the representation
theory of the finite group $G(O_r)$ is less complete. 
Shalika \cite{Shalika2004} treats the case $SL_2(O_r)$, Silberger 
\cite{Silberger1970} the case $PGL_2(O_r)$. Shintani \cite{Shintani1968} 
and G\'erardin \cite{Gerardin1972} treat cuspidal representations
of $GL_n(O_r)$ in order to construct supercuspidal representations of
$GL_n(F)$. The last two papers 
use Clifford theory and Weil representations over
finite fields. In the series of papers 
\cite{Hill1993,Hill1994,Hill1995-1,Hill1995-2}, Hill treats the case
$GL_n(O_r)$ systematically by means of Clifford theory, but 
different methods are used 
for representations associated with different type of adjoint orbits. 

In this paper, we will establish a parametrization of the irreducible
representations of $G(O_r)$ ($r>1$) associated with the regular (more 
precisely smoothly regular) adjoint orbits. Taking a representative 
$\beta$ of the adjoint orbit, the parametrization is given by means of
a subset of the character group of $G_{\beta}(O_r)$ where $G_{\beta}$
is the centralizer of $\beta$ in $G$ which is smooth commutative group
scheme over $O$. Our theory is based on Clifford theory
and Weil representations over finite fields, and 
it works well under an assumption
of the triviality of certain Schur multiplier of a finite commutative 
group $G_{\beta}(\Bbb F)$. We can verify the
assumption in the case of $GL_n$ with $n\leq 4$, and the discussions
in this paper show that the assumption is highly probable for the
reductive groups in general 

The main result of this paper is Theorem \ref{th:main-result}. The
situation is quite simple when $r$ is even, and almost all of this
paper is devoted to treat the case of $r=2l-1$ being odd. In this case
we need Weil representation over finite field to construct irreducible
representations of $G_{\beta}(O_r)\cdot K_{l-1}(O_r)$, 
where $K_{l-1}(O_r)$ is the kernel of the canonical group homomorphism
$G(O_r)\to G(O_{l-1})$, and at this
point appears the Schur multiplier as an obstruction to the
construction. Here we shall note that over a finite
field, Weil representation is a representation of a symplectic group, 
not of the $2$-fold covering group of it, and that the Schur
multiplier is coming not from Weil representation but from 
certain {\it twist} which occurs en route of connecting 
$K_{l-1}(O_r)$ with the Heisenberg group over finite field 
(see section \ref{sec:weil-representation-over-finite-field} for the
details). 

Several fundamental properties of the Schur multiplier will be
discussed in section \ref{sec:properties-of-c-beta-rho}. 
These properties, combined with the
results of \cite{Takase2016} in the case of $G=GL_n$, shows that it is
highly 
probable that the Schur multiplier is trivial for all reductive group
schemes over $O$ provided that $\beta$ is regular and that 
the residue characteristic is big enough. 

We will give some examples of classical groups where the
characteristic polynomial of $\beta$ is irreducible modulo
$\frak{p}$. In this case the
parametrization is given by a subset of the character group of unit
groups of unramified extensions of the base field $F$. See propositions 
\ref{prop:shintani-gerardin-parametrization} for a general linear
group, 
\ref{prop:shintani-gerardin-type-parametrization-for-gsp-2n} for a
group of symplectic similitudes, 
\ref{prop:shintani-gerardin-type-parametrization-for-go(s)-of-even-n} and 
\ref{prop:shintani-gerardin-type-parametrization-for-go(s)-of-odd-n} for a
general orthogonal group with respect to a quadratic form of even and
odd variables respectively, and 
\ref{prop:shintani-gerardin-type-parametrization-for-u(s)} for an
unitary group associated with Hermitian form of odd variables.

\section{Main results}\label{sec:main-results}

\subsection[]{}\label{subsec:fundamental-setting}
Let $G\subset GL_n$ be a closed smooth $O$-group subscheme, and
$\frak{g}$ the Lie algebra of $G$ which is a closed affine
$O$-subscheme of $\frak{gl}_n$ the Lie algebra of $GL_n$. We may
  assume that the fibers $G{\otimes}_OK$ ($K=F$ or $K=\Bbb F$) are
  non-commutative algebraic $K$-group (that is smooth $K$-group
  scheme).  

For any
$O$-algebra $K$, the set of the $K$-valued points $\frak{gl}_n(K)$ is
  identified with the $K$-Lie algebra of square matrices $M_n(K)$ of size $n$
  with Lie bracket $[X,Y]=XY-YX$, and 
the group of $K$-valued points $GL_n(K)$ is
identified with the matrix group
$$
 GL_n(K)=\{g\in M_n(K)\mid \det g\in K^{\times}\}
$$  
where $K^{\times}$ is the multiplicative group of $K$. Hence
$\frak{g}(K)$ is identified with a matrix Lie subalgebra of $\frak{gl}_n(K)$
and $G(K)$ is identified with a matrix subgroup of $GL_n(K)$. Let 
$$
 B:\frak{gl}_n{\times}_O\frak{gl}_n\to
                \mathcal{O}=\text{\rm Spec}(O[t])
$$
be the trace form on $\frak{gl}_n$, that is $B(X,Y)=\text{\rm tr}(XY)$
for all $X,Y\in\frak{gl}_n(K)$ with any $O$-algebra $K$. The
smoothness of $G$ implies that we have a canonical isomorphism
$$
 \frak{g}(O)/\varpi^r\frak{g}(O)\,\tilde{\to}\,
 \frak{g}(O_r)=\frak{g}(O){\otimes}_{O}O_r
$$
(\cite[Chap.II, $\S 4$, Prop.4.8]{Demazure-Gabriel1970}) and that the
 canonical group homomorphism $G(O)\to G(O_r)$ is
 surjective due to Hensel's lemma. Then for any $0<l<r$ the canonical
 group homomorphism $G(O_r)\to G(O_l)$ is surjective whose kernel is
 denoted by $K_l(O_r)$. 

For any $g\in G(O)$ (resp. $X\in\frak{g}(O)$), 
the image under the canonical surjection onto $G(O_l)$ 
(resp. onto $\frak{g}(O_l)$) with $l>0$ is denoted by 
$$
 g_l=g\npmod{\frak{p}^l}\in G(O_l)
 \quad
(\text{\rm resp.} X_l=X\npmod{\frak{p}^l}\in\frak{g}(O_l)).
$$
Since the reduction modulo $\frak{p}$ plays a fundamental role in our
theory, let us use 
the notation $\overline g=g\npmod{\frak p}\in G(\Bbb F)$ 
(resp. $\overline X=X\npmod{\frak p}\in\frak{g}(\Bbb F)$) if $l=1$. 

We will pose the following three conditions;
\begin{itemize}
\item[I)] $B:\frak{g}(\Bbb F)\times\frak{g}(\Bbb F)\to\Bbb F$ is
  non-degenerate, 
\item[II)] for any integers $r=l+l^{\prime}$ with 
           $0<l^{\prime}\leq l<r$, we have a group isomorphism
$$
 \frak{g}(O_{l^{\prime}})\,\tilde{\to}\,K_l(O_r)
$$
           defined by 
$X\npmod{\frak{p}^{l^{\prime}}}\mapsto1+\varpi^lX\npmod{\frak{p}^r}$,
\item[III)] if $r=2l-1\geq 3$ is odd, then we have a mapping
$$
 \frak{g}(O)\to K_{l-1}(O_r)
$$
defined by 
$X\mapsto(1+\varpi^{l-1}X+2^{-1}\varpi^{2l-2}X^{2l-2})\npmod{\frak{p}^r}$.
\end{itemize}
The condition I) implies that 
$B:\frak{g}(O_l)\times\frak{g}(O_l)\to O_l$ 
is non-degenerate for all $l>0$, and so 
$B:\frak{g}(O)\times\frak{g}(O)\to O$ is also non-degenerate. 
The mappings of the conditions II) and III) from Lie algebras to
groups can be regarded as truncations of the exponential mapping.

See section \ref{sec:example} for the examples of classical groups
which satisfy these three fundamental conditions of our theory. 

The character group of an finite abelian group $\mathcal{G}$ 
is denoted by $\mathcal{G}\sphat$.

\subsection[]{}\label{subsec:clifford-theory-in-general}
From now on we will fix an integer $r\geq 2$ and put 
$r=l+l^{\prime}$ with the smallest integer $l$ such that 
$0<l^{\prime}\leq l$. In other word
$$
 l^{\prime}=\begin{cases}
             l&:r=2l,\\
             l-1&:r=2l-1.
            \end{cases}
$$
Take a $\beta\in\frak{g}(O)$ and define a
character $\psi_{\beta}$ of the finite abelian group 
$K_l(O_r)$ by 
$$
 \psi_{\beta}((1+\varpi^lX)\nnpmod{\frak{p}^r})
 =\tau(\varpi^{-l^{\prime}}B(X,\beta))
 \qquad
 (X\in\frak{g}(O)).
$$
Then $\beta\npmod{\frak{p}^{l^{\prime}}}\mapsto\psi_{\beta}$ 
gives an isomorphism of the
additive group $\frak{g}(O_{l^{\prime}})$ onto the character
group $K_l(O_r)\sphat$. For any 
$g_r=g\npmod{\frak{p}^r}\in G(O_r)$, we have
$$
 \psi_{\beta}(g_r^{-1}hg_r)
 =\psi_{\text{\rm Ad}(g)\beta}(h)
 \qquad
 (h\in K_l(O_r)).
$$
So the stabilizer of $\psi_{\beta}$ in $G(O_R)$ is 
$$
 G(O_r,\beta)
 =\left\{g_r\in G(O_r)\bigm|
         \text{\rm Ad}(g)\beta\equiv\beta
                    \pmod{\frak{p}^{l^{\prime}}}\right\}
$$
which is a subgroup of $G(O_r)$ containing $K_{l^{\prime}}(O_r)$.

Now let us denote by 
$\text{\rm Irr}(G(O_r)\mid\psi_{\beta})$ 
(resp. $\text{\rm Irr}(G(O_r,\beta)\mid\psi_{\beta})$) the set of the
isomorphism classes of the irreducible complex representation $\pi$ of 
$G(O_r)$ (resp. $\sigma$ of $G(O_r,\beta)$) such
that 
$$
 \langle\psi_{\beta},\pi\rangle_{K_l(O_r)}
 =\dim_{\Bbb C}\text{\rm Hom}_{K_l(O_r)}(\psi_{\beta},\pi)
 >0
$$
(resp. 
$\langle\psi_{\beta},\sigma\rangle_{K_l(O_r)}>0$). Then Clifford's
 theory says that 
\begin{enumerate}
\item $\text{\rm Irr}(G(O_r))
       =\bigsqcup_{\beta\npmod{\frak{p}^{l^{\prime}}}}
           \text{\rm Irr}(G(O_r)\mid\psi_{\beta})$
  where $\bigsqcup_{\beta\npmod{\frak{p}^{l^{\prime}}}}$ 
  is the disjoint union over the
  representatives $\beta\npmod{\frak{p}^{l^{\prime}}}$ 
  of the $\text{\rm Ad}(G(O_{l^{\prime}}))$-orbits in 
  $\frak{g}(O_{l^{\prime}})$,
\item a bijection of 
      $\text{\rm Irr}(G(O_r,\beta)\mid\psi_{\beta})$ onto 
      $\text{\rm Irr}(G(O_r)\mid\psi_{\beta})$ is given by 
$$
 \sigma\mapsto\text{\rm Ind}_{G(O_r,\beta)}^{G(O_r)}\sigma.
$$
\end{enumerate}
So our problem is to give a good parametrization of the 
   set $\text{\rm Irr}(G(O_r,\beta)\mid\psi_{\beta})$ for 
$\beta\in\frak{g}(O)$ which is {\it regular} enough. 

For any $\beta\in\frak{g}(O)$, let us denote
by $G_{\beta}=Z_G(\beta)$ the centralizer of
$\beta$ in $G$  which is a closed $O$-group subscheme of $G$. The Lie
algebra $\frak{g}_{\beta}=Z_{\frak g}(\beta)$ of $G_{\beta}$ is 
the centralizer of $\beta$ in $\frak{g}$ which is 
a closed $O$-subscheme of $\frak{g}$. 

\subsection[]{}
\label{subsec:def-of-schur-multiplier-associated-with-alg-group}
In this subsection we will define a Schur multiplier which is an
obstruction to our theory.

Take a $\beta\in\frak{g}(O)$ such that 
$\frak{g}_{\beta}(\Bbb F)\lvertneqq\frak{g}(\Bbb F)$. Then non-zero 
$\Bbb F$-vector space 
$\Bbb V_{\beta}=\frak{g}(\Bbb F)/\frak{g}_{\beta}(\Bbb F)$ has a
symplectic form 
$$
 \langle \dot X,\dot Y\rangle_{\beta}=B([X,Y],\overline\beta)
$$
where $\dot X=X\pmod{\frak{g}_{\beta}(\Bbb F)}\in\Bbb V_{\beta}$ with 
$X\in\frak{g}_{\beta}(\Bbb F)$.  
Then $g\in G_{\beta}(\Bbb F)$ gives an element $\sigma_g$ of 
 the symplectic group $Sp(\Bbb V_{\beta})$ defined by 
$$
 X\npmod{\frak{g}_{\beta}(\Bbb F)}\mapsto
 \text{\rm Ad}(g)^{-1}X\npmod{\frak{g}_{\beta}(\Bbb F)}.
$$
Note that the group $Sp(\Bbb V_{\beta})$ acts on $\Bbb V_{\beta}$ from
right. Let 
$v\mapsto[v]$ be a $\Bbb F$-linear section on $\Bbb V_{\beta}$ of the
exact sequence 
\begin{equation}
 0\to\frak{g}_{\beta}(\Bbb F)
  \to\frak{g}(\Bbb F)\to\Bbb V_{\beta}\to 0.
\label{eq:fundamental-exact-sequence-of-lie-alg-and-symplectic-space}
\end{equation}
For any $v\in\Bbb V_{\beta}$ and $g\in G_{\beta}(\Bbb F)$, put
$$
 \gamma(v,g)=\gamma_{\frak g}(v,g)
 =\text{\rm Ad}(g)^{-1}[v]-[v\sigma_g]\in\frak{g}_{\beta}(\Bbb F).
$$
Take a $\rho\in\frak{g}_{\beta}(\Bbb F)\sphat$. Then there exists
uniquely a $v_g\in\Bbb V_{\beta}$ such that
$$
 \rho(\gamma(v,g))=\widehat\tau(\langle v,v_g\rangle_{\beta})
$$
for all $v\in\Bbb V_{\beta}$. Note that $v_g\in\Bbb V_{\beta}$ depends
on $\rho$ as well as the section $v\mapsto[v]$. Let
$$
 G_{\beta}(\Bbb F)^{(c)}
 =\{g\in G(\Bbb F)\mid\text{\rm Ad}(g)Y=Y\;\text{\rm for}\;
                        \forall Y\in\frak{g}_{\beta}(\Bbb F)\}
$$
be the centralizer of $\frak{g}_{\beta}(\Bbb F)$ in $G(\Bbb F)$, which
is a subgroup of $G_{\beta}(\Bbb F)$. Then for any 
$g,h\in G_{\beta}(\Bbb F)^{(c)}$, we have
\begin{equation}
 v_{gh}=v_h\sigma_g^{-1}+v_g
\label{eq:multiplication-formula-of-v-epsilon}
\end{equation}
because $\gamma(v,gh)=\gamma(v,g)+\gamma(v\sigma_g^{-1},h)$ for all
$v\in\Bbb V_{\beta}$. Put
$$
 c_{\beta,\rho}(g,h)
 =\widehat\tau(2^{-1}\langle v_g,v_{gh}\rangle_{\beta})
$$
for $g,h\in G_{\beta}(\Bbb F)^{(c)}$. Then the relation 
\eqref{eq:multiplication-formula-of-v-epsilon} shows that 
$c_{\beta,\rho}\in Z^2(G_{\beta}(\Bbb F)^{(c)},\Bbb C^{\times})$ is a
2-cocycle with trivial action of $G_{\beta}(\Bbb F)^{(c)}$ on 
$\Bbb C^{\times}$. Moreover we have

\begin{prop}
\label{prop:schur-multiplier-associated-with-finite-symplectic-space}
The Schur multiplier 
$[c_{\beta,\rho}]\in H^2(G_{\beta}(\Bbb F)^{(c)},\Bbb C^{\times})$ is
independent of the choice of the $\Bbb F$-linear section 
$v\mapsto[v]$.
\end{prop}
\begin{proof}
Take another $\Bbb F$-linear section $v\mapsto[v]^{\prime}$ 
with respect to which
we will define $\gamma^{\prime}(v,g)\in\frak{g}_{\beta}$ 
and $v^{\prime}_g\in\Bbb V_{\beta}$ as above. 
Then there exists a $\delta\in\Bbb V_{\beta}$ such that 
$\rho([v]-[v]^{\prime})
 =\widehat\tau\left(\langle v,\delta\rangle_{\beta}\right)$ 
for all $v\in\Bbb V_{\beta}$. We have 
$v_g^{\prime}=v_g+\delta-\delta\sigma_g$ for all 
$g\in G_{\beta}(\Bbb F)^{(c)}$. So if we put 
$\alpha(g)
 =\widehat\tau\left(2^{-1}\langle 
    v_g^{\prime}-v_{g^{-1}},\delta\rangle_{\beta}
            \right)$ 
for $g\in G_{\beta}(\Bbb F)^{(c)}$, then we have
$$
 \widehat\tau\left(2^{-1}\langle v_g^{\prime},
           v_{gh}^{\prime}\rangle_{\beta}\right)
 =\widehat\tau\left(2^{-1}\langle v_g,
           v_{gh}\rangle_{\beta}\right)\cdot
  \alpha(h)\alpha(gh)^{-1}\alpha(g)
$$
for all $g,h\in G_{\beta}(\Bbb F)^{(c)}$.
\end{proof}

\subsection[]{}\label{subsec:main-result}
Now our main result is

\begin{thm}\label{th:main-result}
Suppose that a $\beta\in\frak{g}(O)$ satisfies the conditions
\begin{enumerate}
\item $G_{\beta}$ is commutative smooth $O$-group scheme, and
\item the Schur multiplier 
      $[c_{\beta,\rho}]\in H^2(G_{\beta}(\Bbb F),\Bbb C^{\times})$ is
      trivial for all characters $\rho\in\frak{g}_{\beta}(\Bbb F)\sphat$.
\end{enumerate}
Then 
we have a bijection $\theta\mapsto\sigma_{\beta,\theta}$ of the set
$$
 \left\{\theta\in G_{\beta}(O_r)\sphat\;\;\;
         \text{\rm s.t. $\theta=\psi_{\beta}$ on 
               $G_{\beta}(O_r)\cap K_l(O_r)$}\right\}
$$
onto $\text{\rm Irr}(G(O_r,\beta)\mid\psi_{\beta})$.
\end{thm}

The proof is given in subsection
\ref{subsec:proof-of-main-result-for-even-r} for even $r$ and 
subsection \ref{subsec:description-of-main-result-for-odd-r} for odd
$r$.

\begin{rem}\label{remark:remarks-on-main-result}
\begin{enumerate}
\item A sufficient condition for the first condition of Theorem
  \ref{th:main-result} is given by Theorem
  \ref{th:sufficient-condition-for-smooth-commutativeness-of-g-beta}.
\item The smoothness of $G_{\beta}$ over $O$ implies that the
      canonical group homomorphism 
      $G_{\beta}(O_r)\to G_{\beta}(O_{l^{\prime}})$ is surjective. So we
      have
\begin{equation}
 G(O_r,\beta)=G_{\beta}(O_r)\cdot K_{l^{\prime}}(O_r),
 \qquad
 l^{\prime}=\begin{cases}
             l&:r=2l,\\
             l-1&:r=2l-1.
            \end{cases}
\label{eq:fundamental-structure-of-stabilizer-of-psi-beta}
\end{equation}
\item As presented in the following two subsections, 
      the second condition in the theorem is required only in the case
      of $r$ being odd.
\item Since $G{\otimes}_O\Bbb F$ and $G_{\beta}{\otimes}_O\Bbb F$ are
      $\Bbb F$-algebraic group, and the former is not commutative
      while the latter is, so we have
$$
 \dim_{\Bbb F}\frak{g}(\Bbb F)=\dim G{\otimes}_O\Bbb F
 >\dim G_{\beta}{\otimes}_O\Bbb F=\dim_{\Bbb F}\frak{g}_{\beta}(\Bbb F).
$$
      That is $\frak{g}_{\beta}(\Bbb F)\lvertneqq\frak{g}(\Bbb F)$.
\item Since $G_{\beta}$ is assumed to be commutative, we have
      $G_{\beta}(\Bbb F)^{(c)}=G_{\beta}(\Bbb F)$ in the definition of
      the Schur multiplier $[c_{\beta,\rho}]$. 
\item Assume that $G_{\beta}(\Bbb F)^{(c)}$ is commutative. Then the
      cohomology class 
      $[c_{\beta,\rho}]
       \in H^2(G_{\beta}(\Bbb F)^{(c)},\Bbb C^{\times})$ 
      is trivial if and only if $c_{\beta,\rho}$ is symmetric, that is 
      $c_{\eta,\rho}(g,h)=c_{\beta,\rho}(h,g)$ for all 
      $g,h\in G_{\beta}(\Bbb F)^{(c)}$. In fact, only if part is
      trivial. Let
\begin{equation}
 1\to\Bbb T\xrightarrow{i}\mathcal{G}
           \xrightarrow{j}G_{\beta}(\Bbb F)^{(c)}\to 1
\label{eq:group-extension-of-c-beta-rho}
\end{equation}
      be the group extension associated with the $2$-cocycle 
      $c_{\beta,\rho}
       \in Z^2(B_{\beta}(\Bbb F)^{(c)},\Bbb T)$ where $\Bbb T$ is the
       subgroup of $z\in\Bbb C^{\times}$ such that $|z|=1$. 
       Then the
       groups are compact commutative group, and we have a group
       extension of the Pontryagin dual groups
\begin{equation}
 1\to{G_{\beta}(\Bbb F)^{(c)}}\sphat
  \xrightarrow{\widehat j}\mathcal{G}\sphat
  \xrightarrow{\widehat i}{\Bbb T}\sphat\to 1.
\label{eq:dual-of-group-extension-of-c-beta-rho}
\end{equation}
      Since ${\Bbb T}\sphat\simeq\Bbb Z$ is free the group extension 
      \eqref{eq:dual-of-group-extension-of-c-beta-rho} is trivial and so
      is the group extension 
      \eqref{eq:group-extension-of-c-beta-rho}.
\end{enumerate}
\end{rem}

\subsection[]{}\label{subsec:proof-of-main-result-for-even-r}
Assume that $r=2l$ is even so that $l^{\prime}=l$. In this case the
proof of Theorem \ref{th:main-result} is quite easy. Let us suppose
more generally that there exists a commutative subgroup $\mathcal{C}$
of $G(O_r,\beta)$ such that
$$
 G(O_r,\beta)=\mathcal{C}\cdot K_l(O_r).
$$
Let us denote by $\mathcal{C}\sphat_{\beta}$ the subset of the
character group $\mathcal{C}\sphat$ consisting of the
$\theta\in\mathcal{C}\sphat$ such 
that $\theta=\psi_{\beta}$ on $\mathcal{C}\cap K_l(O_r)$. Then any 
$\theta\in\mathcal{C}\sphat_{\beta}$ gives an one-dimensional
representation $\sigma_{\beta,\theta}$ of $G(O_r,\beta)$ defined by 
$$
 \sigma_{\beta,\theta}(gh)=\theta(g)\cdot\psi_{\beta}(h)
 \qquad
 (g\in\mathcal{C}, h\in K_l(O_r)).
$$
Then we have a proposition of which our Theorem 
\ref{th:main-result} is a special case;

\begin{prop}\rm\label{prop:generalized-main-result-for-even-r}
$\theta\mapsto\sigma_{\beta,\theta}$ gives a bijection of 
$\mathcal{C}\sphat_{\beta}$ onto 
$\text{\rm Irr}(G(O_r,\beta)\mid\psi_{\beta})$.
\end{prop}
\begin{proof}
Take any $\sigma\in\text{\rm Irr}(G(O_r,\beta)\mid\psi_{\beta})$ with
representation space $V_{\sigma}$. Then
$$
 V_{\sigma}(\psi_{\beta})
 =\{v\in V_{\sigma}\mid \sigma(g)v=\psi_{\beta}(g)v\,
                \text{\rm for $\forall g\in K_l(O_r)$}\}
$$
is a non-trivial $G(O_r,\beta)$-subspace of $V_{\sigma}$ so that 
$V_{\sigma}=V_{\sigma}(\psi_{\beta})$. Then, for any 
one-dimensional representation $\chi$ of $G(O_r,\beta)$ such that 
$\chi=\psi_{\beta}$ on $K_l(O_r)$, we have 
$K_l(O_r)\subset\text{\rm Ker}(\chi^{-1}\otimes\sigma)$. On the other
hand $G(O_r,\beta)/K_l(O_r)$ is commutative, we have 
$\dim(\chi^{-1}\otimes\sigma)=1$ and then $\dim\sigma=1$. 
Put $\theta=\sigma|_{\mathcal{C}}\in\mathcal{C}\sphat_{\beta}$ and 
we have $\sigma=\sigma_{\beta,\theta}$. 
\end{proof}

\subsection[]{}\label{subsec:description-of-main-result-for-odd-r}
Assume that $r=2l-1\geq 3$ is odd so that $l^{\prime}=l-1\geq 1$. 
We have a chain of canonical surjections
\begin{equation}
 \heartsuit:K_{l-1}(O_r)\to K_{l-1}(O_{r-1})\,\tilde{\to}\,
        \frak{g}(O_{l-1})\to\frak{g}(\Bbb F)
\label{eq:canonical-surjection-of-k-l-1-to-g(f)}
\end{equation}
defined by
\begin{align*}
 1+\varpi^{l-1}X\nnpmod{\frak{p}^r}
 &\mapsto
 1+\varpi^{l-1}X\nnpmod{\frak{p}^{r-1}}\\
 &\mapsto
  X\nnpmod{\frak{p}^{l-1}}\mapsto
  \overline X=X\npmod{\frak p}..
\end{align*}
Let us denote by $Z(O_r,\beta)$ the inverse image 
under the surjection $\heartsuit$ of $\frak{g}_{\beta}(\Bbb F)$. Then 
$Z(O_r,\beta)$ is a normal subgroup of $K_{l-1}(O_r)$ containing 
$K_l(O_r)$ as the kernel of $\heartsuit$. 

Let us denote by $Y_{\beta}$ the set of the group
homomorphisms $\psi$ of $Z(O_r,\beta)$ to $\Bbb C^{\times}$ 
such that $\psi=\psi_{\beta}$ on $K_l(O_r)$. Then
a bijection of $\frak{g}_{\beta}(\Bbb F)\sphat$ onto $Y_{\beta}$ is
given by
$$
 \rho\mapsto\psi_{\beta,\rho}
            =\widetilde{\psi}_{\beta}\cdot(\rho\circ\heartsuit),
$$
where a group homomorphism 
$\widetilde{\psi}_{\beta}:Z(O_r,\beta)\to\Bbb C^{\times}$ 
is defined by 
$$
 1+\varpi^{l-1}X\npmod{\frak{p}^r}\mapsto
 \tau\left(\varpi^{-l}B(X,\beta)-(2\varpi)^{-1}B(X^2,\beta)\right)
$$
with $\overline X=X\npmod{\frak{p}}\in\frak{g}_{\beta}(\Bbb F)$. 

Take a $\psi\in Y_{\beta}$. 
For two elements 
$$
 x=1+\varpi^{l-1}X\npmod{\frak{p}^r},
 \quad
 y=1+\varpi^{l-1}Y\npmod{\frak{p}^r}
$$ 
of $K_{l-1}(O_r)$, we have 
$x^{-1}=1-\varpi^{l-1}X+2^{-1}\varpi^{2l-2}X^2\npmod{\frak{p}^r}$ so
that we have
$$
 xyx^{-1}y^{-1}
 =1+\varpi^{r-1}[X,Y]\npmod{\frak{p}^r}
 \in K_{r-1}(O_r)\subset K_l(O_r)
$$
and so 
$\psi_{\beta}(xyx^{-1}y^{-1})
 =\tau\left(\varpi^{-1}B(X,\text{\rm ad}(Y)\beta)\right)$. 
Hence we have 
$$
 \psi(xyx^{-1}y^{-1})=\psi_{\beta}(xyx^{-1}y^{-1})=1
$$
for all $x\in K_{l-1}(O_r)$ and $y\in Z(O_r,\beta)$ so
that we can define 
$$
 D_{\psi}:K_{l-1}(O_r)/Z(O_r,\beta)\times K_{l-1}(O_r)/Z(O_r,\beta)
          \to\Bbb C^{\times}
$$
by
$$
 D_{\psi}(\dot g,\dot h)
 =\psi(ghg^{-1}h^{-1})
 =\psi_{\beta}(ghg^{-1}h^{-1})
 =\tau\left(\varpi^{-1}B([X,Y],\beta)\right)
$$
for 
$g=(1+\varpi^{l-1}X)\npmod{\frak{p}^r}, 
 h=(1+\varpi^{l-1}Y)\npmod{\frak{p}^r}
 \in K_{l-1}(O_r)$, which is non-degenerate. Then Proposition 3.1.1 of 
\cite{Takase2016} gives

\begin{prop}
\label{prop:existence-of-fundamental-rep-associated-with-psi} 
For any $\psi=\psi_{\beta,\rho}\in Y_{\beta}$ with 
$\rho\in\frak{g}_{\beta}(\Bbb F)\sphat$, 
there exists unique irreducible
representation $\pi_{\beta,\rho}$ of $K_{l-1}(O_r)$ such that 
$\langle\psi,\pi_{\beta,\rho}\rangle_{Z(O_r,\beta)}>0$. Furthermore 
$$
 \text{\rm Ind}_{Z(O_r,\beta)}^{K_{l-1}(O_r)}\psi
 =\bigoplus^{\dim\pi_{\beta,\rho}}\pi_{\beta,\rho}
$$
and $\pi_{\beta,\rho}(x)$ is the homothety $\psi(x)$ for all 
$x\in Z(O_r,\beta)$.
\end{prop}

Fix a $\psi=\psi_{\beta,\rho}\in Y_{\beta}$ with 
$\rho\in\frak{g}_{\beta}(\Bbb F)\sphat$. 
Let $G_l(O_r,\beta)^{(c)}$ be a subgroup of $G(O_r,\beta)$ defined
by
$$
 G_l(O_r,\beta)^{(c)}
 =\left\{g\npmod{\frak{p}^r}\in G(O_r)\biggm|
    \begin{array}{l}
     \text{\rm Ad}(g)\beta\equiv\beta\npmod{\frak{p}^l},\\
     \text{\rm Ad}(\overline g)X=X\,\text{\rm for}\,
      \forall X\in\frak{g}_{\beta}(\Bbb F)
    \end{array}\right\}.
$$
Then, for any $g_r=g\npmod{\frak{p}^r}\in G_l(O_r,\beta)^{(c)}$ and 
$x=(1+\varpi^{l-1}X)\npmod{\frak{p}^r}\in Z(O_r,\beta)$, we have
\begin{align*}
 g_r^{-1}xg_rx^{-1}
 &=\left(1+\varpi^{l-1}g^{-1}Xg\right)
  \left(1-\varpi^{l-1}X+2^{-1}\varpi^{2l-2}X^2\right)\npmod{\frak{p}^r}\\
 &=1+\varpi^{l-1}\left(\text{\rm Ad}(g)^{-1}X-X\right)
    \npmod{\frak{p}^r}
  \in K_l(O_r),
\end{align*}
and
$$
 \psi(g_r^{-1}xg_rx^{-1})
 =\psi_{\beta}(g_r^{-1}xg_rx^{-1})
 =\tau\left(\varpi^{-l}B(X,\text{\rm Ad}(g)\beta-\beta)\right)=1,
$$
that is $\psi(g_r^{-1}xg_r)=\psi(x)$ for all $x\in Z(O_r,\beta)$. This
means that, for any $g\in G_l(O_r,\beta)^{(c)}$, the $g$-conjugate of
$\pi_{\beta,\rho}$ is isomorphic to $\pi_{\beta,\rho}$, that is, there
exists a $U(g)\in GL_{\Bbb C}(V_{\beta,\rho})$ ($V_{\beta,\rho}$ is the
representation space of $\pi_{\beta,\rho}$) such that
$$
 \pi_{\beta,\rho}(g^{-1}xg)
 =U(g)^{-1}\circ\pi_{\beta,\rho}(x)\circ U(g)
$$
for all $x\in K_{l-1}(O_r)$, and moreover, for any 
$g,h\in G_l(O_r,\beta)^{(c)}$, there exists a 
$c_U(g,h)\in\Bbb C^{\times}$ such that
$$
 U(g)\circ U(h)=c_U(g,h)\cdot U(gh).
$$
Then $c_U\in Z^2(G_l(O_r,\beta)^{(c)},\Bbb C^{\times})$ is a 
$\Bbb C^{\times}$-valued 2-cocycle on $G_l(O_r,\beta)^{(c)}$ with
trivial action on $\Bbb C^{\times}$, and the cohomology class 
$[c_U]\in H^2(G_l(O_r,\beta)^{(c)},\Bbb C^{\times})$ is independent of
the choice of each $U(g)$. 

In the next section, we will construct $\pi_{\beta,\rho}$ by means of
Weil representations over the finite field $\Bbb F$ (see Proposition
\ref{prop:another-expression-of-pi-beta-psi}), and will show
that we can choose $U(g)$ so that we have 
$$
 c_U(g,h)=c_{\beta,\rho}(\overline g,\overline h)
$$ 
for all $g,h\in G_l(O_r,\beta)^{(c)}$, where 
$\overline g\in G_{\beta}(\Bbb F)^{(c)}$ is the image of 
$g\in G_l(O_r,\beta)^{(c)}$ under the canonical surjection 
$G(O_r)\to G(\Bbb F)$ (see subsection \ref{subsec:action-of-g-o-f-beta}).

Let us assume 

\begin{assu}
\label{hypo:existence-of-good-abelian-subgroup-of-g1(c)} 
There exists a commutative subgroup 
$\mathcal{C}\subset G_l(O_r,\beta)^{(c)}$ such that
\begin{enumerate}
\item $G(O_r,\beta)=\mathcal{C}\cdot K_{l-1}(O_r)$,
\item the cohomology class 
      $[c_{\beta,\rho}|_{\overline{\mathcal{C}}\times\overline{\mathcal{C}}}]
       \in H^2(\overline{\mathcal{C}},\Bbb C^{\times})$ is trivial for all 
      $\rho\in\frak{g}_{\beta}(\Bbb F)\sphat$, where
       $\overline{\mathcal C}\subset G_{\beta}(\Bbb F)^{(c)}$ is the
       image of $\mathcal{C}$ under the canonical surjection 
      $G(O_r)\to G(\Bbb F)$.
\end{enumerate}
\end{assu}

Under this assumption we have

\begin{prop}
\label{prop:existence-of-canonical-intertwiner-of-pi-psi}
For any $\rho\in\frak{g}_{\beta}(\Bbb F)\sphat$, there exists a group
homomorphism 
$U_{\beta,\rho}:\mathcal{C}\to GL_{\Bbb C}(V_{\beta,\rho})$ such that
\begin{enumerate}
\item $\pi_{\beta,\rho}(g^{-1}xg)
       =U_{\beta,\rho}(g)^{-1}\circ\pi_{\beta,\rho}(x)
                              \circ U_{\beta,\rho}(g)$ for all 
      $g\in\mathcal{C}$ and $x\in K_{l-1}(O_r)$ and
\item $U_{\beta,\rho}(h)=1$ for all $h\in\mathcal{C}\cap K_{l-1}(O_r)$.
\end{enumerate}
\end{prop}
\begin{proof}
Because of 2) in Assumption
\ref{hypo:existence-of-good-abelian-subgroup-of-g1(c)} there exists a
group homomorphism $U:\mathcal{C}\to GL_{\Bbb C}(V_{\beta,\rho})$ such that 
$\pi_{\beta,\rho}(g^{-1}xg)=U(g)^{-1}\circ\pi_{\beta,\rho}(x)\circ U(g)$ for all 
$g\in\mathcal{C}$ and $x\in K_{l-1}(O_r)$. Then for any 
$h\in\mathcal{C}\cap K_{l-1}(O_r)$ there exists a 
$c(h)\in\Bbb C^{\times}$ such that $U(h)=c(h)\cdot\pi_{\beta,\rho}(h)$. 
On the other hand we have 
$$
 G_l(O_r,\beta)^{(c)}\cap K_{l-1}(O_r)\subset Z(O_r,\beta)
$$
since 
$(1+\varpi^{l-1}X)_r\in G_l(O_r,\beta)^{(c)}\cap K_{l-1}(O_r)$
means that
\begin{align*}
 \beta
 &\equiv(1+\varpi^{l-1}X)\beta(1+\varpi^{l-1}X)^{-1}\npmod{\frak{p}^l}\\
 &\equiv(\beta+\varpi^{l-1}X\beta)(1-\varpi^{l-1}X)\npmod{\frak{p}^l}\\
 &\equiv\beta+\varpi^{l-1}[X,\beta]\npmod{\frak{p}^l}
\end{align*}
and then $[X,\beta]\equiv 0\npmod{\frak p}$, that is 
$X\npmod{\frak p}\in\frak{g}_{\beta}(\Bbb F)$. Then 
$\pi_{\beta,\rho}(h)$ is the homothety $\psi_{\beta,\rho}(h)$ for all 
$h\in\mathcal{C}\cap K_{l-1}(O_r)$. Extend the group homomorphism 
$h\mapsto c(h)\psi_{\beta,\rho}(h)$ of $\mathcal{C}\cap K_{l-1}(O_r)$ to a group
homomorphism $\theta:\mathcal{C}\to\Bbb C^{\times}$. Then 
$g\mapsto U_{\psi}(g)=\theta(g)^{-1}U(g)$ is the required group homomorphism.
\end{proof}

Let us denote by 
$\mathcal{C}\sphat{\times}_{K_{l-1}(O_r)}\frak{g}_{\beta}(\Bbb F)\sphat$
the set of 
$(\theta,\rho)
 \in\mathcal{C}\sphat\times\frak{g}_{\beta}(\Bbb F)\sphat$ such that 
$\theta=\psi_{\beta,\rho}$ on $\mathcal{C}\cap K_{l-1}(O_r)$. Then 
$(\theta,\rho)\in
 \mathcal{C}\sphat{\times}_{K_{l-1}(O_r)}\frak{g}_{\beta}(\Bbb F)\sphat$ 
defines an irreducible representation $\sigma_{\theta,\rho}$ of 
$G(O_r,\beta)=\mathcal{C}\cdot K_{l-1}(O_r)$ by
$$
 \sigma_{\theta,\rho}(gh)
 =\theta(g)\cdot U_{\beta,\rho}(g)\circ\pi_{\beta,\rho}(h)
$$
for $g\in\mathcal{C}$ and $h\in K_{l-1}(O_r)$. Then we have

\begin{prop}\label{prop:generalized-main-result-for-odd-r}
Under Assumption
\ref{hypo:existence-of-good-abelian-subgroup-of-g1(c)}, a bijection of 
${\mathcal C}\sphat{\times}_{K_{l-1}(O_r)}\frak{g}_{\beta}(\Bbb F)\sphat$
onto $\text{\rm Irr}(G(O_r,\beta)\mid\psi_{\beta})$ is given by 
$(\theta,\rho)\mapsto\sigma_{\theta,\rho}$.
\end{prop}
\begin{proof}
Clearly 
$\pi_{\theta,\rho}\in\text{\rm Irr}(G(O_r,\beta)\mid\psi_{\beta})$ for
all 
$(\theta,\rho)\in
 {\mathcal C}\sphat{\times}_{K_{l-1}(O_r)}\frak{g}_{\beta}(\Bbb F)\sphat$. 
Take a $\sigma\in\text{\rm Irr}(G(O_r,\beta)\mid\psi_{\beta})$. Then 
\begin{align*}
 \sigma\hookrightarrow
 \text{\rm Ind}_{K_l(O_r)}^{G(O_r,\beta)}\psi_{\beta}
 &=\text{\rm Ind}_{Z(O_r,\beta)}^{G(O_r,\beta)}
   \left(\text{\rm Ind}_{K_l(O_r)}^{Z(O_r,\beta)}\psi_{\beta}
         \right)\\
 &=\bigoplus_{\rho\in\frak{g}_{\beta}(\Bbb F)\sphat}
   \text{\rm Ind}_{Z(O_r,\beta)}^{G(O_r,\beta)}\psi_{\beta,\rho}
\end{align*}
so that there exists a $\rho\in\frak{g}_{\beta}(\Bbb F)\sphat$ such
that 
\begin{align*}
 \sigma\hookrightarrow
    \text{\rm Ind}_{Z(O_r,\beta)}^{G(O_r,\beta)}\psi_{\beta,\rho}
 &=\text{\rm Ind}_{K_{l-1}(O_r)}^{G(O_r,\beta)}
   \left(\text{\rm Ind}_{Z(O_r,\beta)}^{K_{l-1}(O_r)}\psi_{\beta,\rho}
         \right)\\
 &=\bigoplus^{\dim\pi_{\beta,\rho}}
   \text{\rm Ind}_{K_{l-1}(O_r)}^{G(O_r,\beta)}\pi_{\beta,\rho}
  =\bigoplus^{\dim\pi_{\beta,\rho}}\bigoplus_{\theta}\sigma_{\theta,\psi},
\end{align*}
where $\bigoplus_{\theta}$ is the direct sum over 
$\theta\in\mathcal{C}\sphat$ such that $\theta=\psi_{\beta,\rho}$ on 
$\mathcal{C}\cap K_{l-1}(O_r)$. 
Then we have $\sigma=\sigma_{\theta,\rho}$ for some 
$(\theta,\rho)\in
 {\mathcal C}\sphat{\times}_{K_{l-1}(O_r)}\frak{g}_{\beta}(\Bbb F)\sphat$.
\end{proof}

Under the conditions of Theorem \ref{th:main-result}, we can put 
$\mathcal{C}=G_{\beta}(O_r)$. We have the following proposition
by which our Theorem \ref{th:main-result} is given as a special case
of Proposition \ref{prop:generalized-main-result-for-odd-r}.

\begin{prop}\label{prop:fundamental-bijection-of-parameter}
If $G_{\beta}$ is commutative smooth over $O$, then 
$(\theta,\rho)\mapsto\theta$ gives a bijection of 
$G_{\beta}(\Bbb F)\sphat{\times}_{K_{l-1}(O_r)}\frak{g}_{\beta}(\Bbb F)\sphat$ 
onto the set
$$
  \left\{\theta\in G_{\beta}(O_r)\sphat\;\;\;
         \text{\rm s.t. $\theta=\psi_{\beta}$ on 
               $G_{\beta}(O_r)\cap K_l(O_r)$}\right\}.
$$
\end{prop}
\begin{proof}
Take a 
$(\theta,\rho)\in 
 G_{\beta}(\Bbb F)\sphat{\times}_{K_{l-1}(O_r)}\frak{g}_{\beta}(\Bbb F)\sphat$.
The smoothness of $G_{\beta}$ over $O$ implies that the canonical
mapping $\frak{g}_{\beta}(O)\to\frak{g}_{\beta}(\Bbb F)$ is
surjective. So Take a $\overline X\in\frak{g}_{\beta}(\Bbb F)$ with 
$X\in\frak{g}_{\beta}(O)$. Then we have
$$
 g=1+\varpi^{l-1}X+2^{-1}\varpi^{2l-2}X^2\npmod{\frak{p}^r}
 \in K_{l-1}(O_r)\cap G_{\beta}(O_r)
$$
so that
\begin{align*}
 \theta(g)=\psi_{\beta,\rho}(g)
 &=\tau\left(\varpi^{-l}B(X+2^{-1}\varpi^{l-1}X^2,\beta)
             -2^{-1}\varpi^{-1}B(X,\beta)\right)\cdot
   \rho(\overline X)\\
 &=\tau\left(\varpi^{-l}B(X,\beta)\right)\cdot\rho(\overline X).
\end{align*}
Hence we have 
$$
 \rho(\overline X)
 =\tau\left(-\varpi^{-l}B(X,\beta)\right)\cdot
  \theta\left(1+\varpi^{l-1}X+2^{-1}\varpi^{2l-2}X^2\npmod{\frak{p}^r}
              \right).
$$
This means that the mapping $(\theta,\rho)\mapsto\theta$ is
injective. Take $X, X^{\prime}\in\frak{g}_{\beta}(O)$ such that 
$X\equiv X^{\prime}\pmod{\frak p}$. Then we have 
$X^{\prime}=X+\varpi T$ with $T\in\frak{g}_{\beta}(O)$ and
\begin{align*}
 &1+\varpi^{l-1}X^{\prime}
                   +2^{-1}\varpi^{2l-2}X^{\prime 2}\pmod{\frak{p}^r}\\
 =&1+\varpi^{-1}X+2^{-1}\varpi^{2l-2}X^2+\varpi^lT\pmod{\frak{p}^r}\\
 =&(1+\varpi^{l-1}X+2^{-1}\varpi^{2l-2}X^2)(1+\varpi^lT)
         \pmod{\frak{p}^r},
\end{align*}
where $1+\varpi^lT\pmod{\frak{p}^r}\in K_l(O_r)$ and hence
$$
 \theta(1+\varpi^lT\npmod{\frak{p}^r})
 =\psi_{\beta}(1+\varpi^lT\npmod{\frak{p}^r})
 =\tau\left(\varpi^{-(l-1)}B(T,\beta)\right).
$$
This and the commutativity of $G_{\beta}$ show that
$$
 \rho(\overline X)
 =\tau\left(-\varpi^{-l}B(X,\beta)\right)\cdot
  \theta\left(1+\varpi^{l-1}X+2^{-1}\varpi^{2l-2}X^2\npmod{\frak{p}^r}
              \right)
$$
with $\overline X\in\frak{g}_{\beta}(\Bbb F)$ with
$X\in\frak{g}_{\beta}(O)$ gives an well-defined group homomorphism of 
$\frak{g}_{\beta}(\Bbb F)$ to $\Bbb C^{\times}$. Then 
$(\theta,\rho)\in 
 G_{\beta}(O_r)\sphat{\times}_{K_{l-1}(O_r)}\frak{g}_{\beta}(\Bbb F)\sphat$ 
and our mapping in question is surjective.
\end{proof}
         
\subsection[]{}
\label{subsec:suff-condition-for-smooth-commutativeness-of-g-beta}
We will give a sufficient condition on $\beta\in\frak{g}(O)$ under which 
$G_{\beta}$ is commutative and smooth over $O$. 

Let us assume that the connected $O$-group scheme $G$ is reductive,
that is, the fibers $G{\otimes}_OK$ ($K=F,\Bbb F$) are reductive
$K$-algebraic groups. In this case the dimension of the maximal torus
in $G{\otimes}_OK$ is independent of $K$ which is denoted by 
$\text{\rm rank}(G)$. For any $\beta\in\frak{g}(O)$ we have
$$
 \dim_K\frak{g}_{\beta}(K)=\dim\frak{g}_{\beta}
 \geq\dim G_{\beta}{\otimes}K\geq\text{\rm rank}(G).
$$
We say $\beta$ to be {\it smoothly regular} with respect to $G$ over
$K$ (or $\overline\beta\in\frak{g}(K)$ is smoothly regular with
respect to $G{\otimes}_OK$) 
if $\dim_K\frak{g}_{\beta}(K)=\text{\rm rank}(G)$ 
(see \cite[1.4]{Springer1971}). In this case $G_{\beta}{\otimes}_OK$
is smooth over $K$. If $\beta$ is smoothly regular with respect to $G$
over $F$ and over $\Bbb F$, then $\beta$ is said to be smoothly regular
with respect to $G$. 

We say $\beta$ to be {\it connected} with respect to $G$ if the fibers 
$G_{\beta}{\otimes}_OK$ ($K=F, \Bbb F$) are connected. 

\begin{thm}
\label{th:sufficient-condition-for-smooth-commutativeness-of-g-beta}
If $\beta\in\frak{g}(O)$ is smoothly regular and connected with
respect to $G$, then $G_{\beta}$ is commutative and smooth over $O$.
\end{thm}
\begin{proof}
Let $G_{\beta}^o$ be the neutral component of $O$-group scheme
$G_{\beta}$ which is a group functor of the category of $O$-scheme 
(see $\S 3$ of Expos\'e $\text{\rm VI}_B$ in \cite{SGA-3}).
The following statements are equivalent;
\begin{enumerate}
\item $G_{\beta}^o$ is representable as an smooth open $O$-group subscheme of
      $G_{\beta}$, 
\item $G_{\beta}$ is smooth at the points of unit section,
\item each fibers $G_{\beta}{\otimes}_OK$ ($K=F,\Bbb F$) are smooth
      over $K$ and their dimensions are constant
\end{enumerate}
(see Th. 3.10 and Cor. 4.4 of \cite{SGA-3}). So if $\beta$ is
smoothly regular with respect to $G$, then $G_{\beta}^o$ is smooth
open $O$-group subscheme of $G_{\beta}$. If further $\beta$ is connected
with respect to $G$, then $G_{\beta}^o=G_{\beta}$ is smooth over
$O$. Let $\beta=\beta_s+\beta_n$ be the Jordan decomposition of 
$\beta\in\frak{g}_{\beta}(F)$. Then the identity component 
$\mathcal{G}$ of the centralizer $Z_{G{\otimes}_OF}(\beta_s)$ is a
reductive $F$-algebraic group and 
$$
 G_{\beta}{\otimes}_OF=Z_{\mathcal{G}}(\beta_n)
$$
because $G_{\beta}{\otimes}_OF$ is connected. Then \cite{Lou1968}
shows that $G_{\beta}(\overline F)$ is commutative ($\overline F$ is
the algebraic closure of $F$), ans hence $G_{\beta}$ is a commutative
$O$-scheme. 
\end{proof}

\subsection[]{}\label{subsec:conjectures-on-triviality-of-c-beta-rho}
With the detailed discussion given in the section
\ref{sec:properties-of-c-beta-rho} and the results of 
\cite{Takase2016}, the truth of the following statement is highly
probable;
\begin{quote}
Assume that $G$ is connected smooth reductive $O$-group scheme. If 
$\beta\in\frak{g}(O)$ is smoothly regular with respect to $G$ over
$\Bbb F$, then the Schur multiplier 
$[c_{\beta,\rho}]\in H^2(G_{\beta}(\Bbb F)^{(c)},\Bbb C^{\times})$ is
trivial for all $\rho\in\frak{g}_{\beta}(\Bbb F)\sphat$ 
provided that the characteristic of $\Bbb F$ is big enough.
\end{quote}

\section{Weil representations over finite field}
\label{sec:weil-representation-over-finite-field}
In this section we will use the notations of the preceding sections
and will suppose $r=2l-1\geq 3$ is odd so that $l^{\prime}=l-1>0$. 
Fix a 
$\beta\in\frak{g}(O)$ such that $\overline\beta\in\frak{g}(\Bbb F)$ is
not in the center of $\frak{g}(\Bbb F)$. 

\subsection[]{}
\label{subsec:k-l-1-as-group-extension}
A group extension
\begin{equation}
 0\to\frak{g}(O_{l-1})\xrightarrow{\diamondsuit}
     K_{l-1}(O_r)\xrightarrow{\heartsuit}
     \frak{g}(\Bbb F)\to 0
\label{eq:k-l-1-as-group-extension}
\end{equation}
is given by the canonical surjection 
\eqref{eq:canonical-surjection-of-k-l-1-to-g(f)}, whose kernel is
$K_l(O_r)$, with the group isomorphism 
$$
 \diamondsuit:\frak{g}(O_{l-1})\,\tilde{\to}\,K_l(O_r)
$$
defined by 
$S\npmod{\frak{p}^{l-1}}\mapsto(1+\varpi^lS)\npmod{\frak{p}^r}$. 

In order to determine the 2-cocycle of the group extension 
\eqref{eq:k-l-1-as-group-extension}, choose any mapping 
$\lambda:\frak{g}(\Bbb F)\to\frak{g}(O)$ such that 
$X=\lambda(X)\npmod{\frak p}$ for all $X\in\frak{g}(\Bbb F)$ and 
$\lambda(0)=0$, and define a section
$$
 l:\frak{g}(\Bbb F)\to K_{l-1}(O_r)
$$
of \eqref{eq:k-l-1-as-group-extension} by 
$X\mapsto 1+\varpi^{l-1}\lambda(X)
           +2^{-1}\varpi^{2l-2}\lambda(X)^2\npmod{\frak{p}^r}$. 
Then we have
$$
 l(X)^{-1}
 =1-\varpi^{l-1}\lambda(X)+2^{-1}\varpi^{2l-2}\lambda(X)^2
  \npmod{\frak{p}^r}
$$
for all $X\in\frak{g}(\Bbb F)$ and
$$
 l(X)(1+\varpi^lS)l(X)^{-1}
 \equiv 1+\varpi^lS\npmod{\frak{p}^r}
$$
for all $S_{l-1}\in\frak{g}(O_{l-1})$. Furthermore we have
$$
 l(X)l(Y)l(X+Y)^{-1}
 =1+\varpi^l\left(\mu(X,Y)+2^{-1}\varpi^{l-2}[\lambda(X),\lambda(Y)]
                  \right)\npmod{\frak{p}^r}
$$
for all $X,Y\in\frak{g}(\Bbb F)$ where 
$\mu:\frak{g}(\Bbb F)\times\frak{g}(\Bbb F)\to\frak{g}(O)$ 
is defined by 
$$
 \lambda(X)+\lambda(Y)-\lambda(X+Y)
 =\varpi\cdot\mu(X,Y)
$$
for all $X,Y\in\frak{g}(\Bbb F)$. Now we have two elements 
($2$-cocycle) 
$$
 \mu=[(X,Y)\mapsto\mu(X,Y)_{l-1}],
 \quad
 c=[(\overline X,\overline Y)\mapsto 2^{-1}\varpi^{l-2}[X,Y]_{l-1}]
$$
of $Z^2(\frak{g}(\Bbb F),\frak{g}(O_{l-1}))$ 
with trivial action of $\frak{g}(\Bbb F)$ on $\frak{g}(O_{l-1})$. 

We will consider two groups $\Bbb M$ and $\Bbb G$ corresponding to the two
2-cocycles $\mu$ and $c$ respectively. That is the group operation on 
$\Bbb M=\frak{g}(\Bbb F)\times\frak{g}(O_{l-1})$ is defined by
$$
 (X,S_{l-1})\cdot(Y,T_{l-1})=(X+Y,(S+T+\mu(X,Y))_{l-1})
$$
and the group operation on 
$\Bbb G=\frak{g}(\Bbb F)\times\frak{g}(O_{l-1})$ is defined by 
$$
 (\overline X,S_{l-1})\cdot(\overline Y,T_{l-1})
 =((X+Y)\sphat,(S+T+2^{-1}\varpi^{l-2}[X,Y])_{l-1}).
$$
Let $\Bbb G{\times}_{\frak{g}(\Bbb F)}\Bbb M$ be the fiber product of 
$\Bbb G$ and $\Bbb M$ with respect to the canonical projections 
onto $\frak{g}(\Bbb F)$. In other word 
$$
 \Bbb G{\times}_{\frak{g}(\Bbb F)}\Bbb M
 =\left\{(X;S,T)=((X,S),(X,T))\in\Bbb G\times\Bbb M\right\}
$$
is a subgroup of the direct product $\Bbb G\times\Bbb M$. 
We have a surjective group homomorphism 
\begin{equation}
 (\ast):\Bbb G{\times}_{\frak{g}(\Bbb F)}\Bbb M\to K_{l-1}(O_r)
\label{eq:fundamental-surjection-for-general-alg-group}
\end{equation}
defined by 
\begin{align*}
 (X;S_{l-1},T_{l-1})\mapsto&
                      l(X)\cdot(1+\varpi^l(S+T)\npmod{\frak{p}^r})\\
         &=1+\varpi^{l-1}\lambda(X)
            +2^{-1}\varpi^{2l-2}\lambda(X)^2
            +\varpi^l(S+T)\npmod{\frak{p}^r}.
\end{align*}

\subsection[]{}\label{subsec:reduction-to-intermediate-group}
A group homomorphism of additive groups 
$B_{\beta}:\frak{g}(O_{l-1})\to O_{l-1}$ ($X\mapsto B(X,\beta_{l-1})$)
induces a group homomorphism 
$B_{\beta}^{\ast}:H^2(\frak{g}(\Bbb F),\frak{g}(O_{l-1}))\to 
                  H^2(\frak{g}(\Bbb F),O_{l-1})$. Let us denote by 
$\mathcal{H}_{\beta}$ the group associated with the 2-cocycle 
$$
 c_{\beta}=B_{\beta}\circ c
 =\left[(\overline X,\overline Y)\mapsto
        2^{-1}\varpi^{l-2}B([X,Y],\beta)_{l-1}\right]
 \in Z^2(\frak{g}(\Bbb F),O_{l-1}).
$$
That is $\mathcal{H}_{\beta}=\frak{g}(\Bbb F)\times O_{l-1}$ with a
group operation 
$$
 (\overline X,s)\cdot(\overline Y,t)
 =((X+Y)\sphat,s+t+2^{-1}\varpi^{l-2}B([X,Y],\beta)_{l-1}).
$$
Then the center of $\mathcal{H}_{\beta}$ is 
$Z(\mathcal{H}_{\beta})=\frak{g}_{\beta}(\Bbb F)\times O_{l-1}$, the
direct product of two additive groups $\frak{g}_{\beta}(\Bbb F)$ and
$O_{l-1}$. 

The inverse
image of $Z(\mathcal{H}_{\beta})$ with respect to the surjective group
homomorphism 
\begin{equation}
 \spadesuit:\Bbb G{\times}_{\frak{g}(\Bbb F)}\Bbb M\to\mathcal{H}_{\beta}
 \qquad
 \left((X;S_{l-1},T_{l-1})\mapsto
       (X,B(S,\beta)_{l-1})\right)
\label{eq:fund-surjection-on-hesenberg-like-group-for-general-alg-group}
\end{equation}
is 
$\left(\Bbb G{\times}_{\frak{g}(\Bbb F)}\Bbb M\right)_{\beta}
 =\left\{(X;S,T)\in\Bbb G{\times}_{\frak{g}(\Bbb F)}\Bbb M\mid
         X\in\frak{g}(\Bbb F)_{\overline\beta}\right\}$ 
which is 
mapped onto $Z(O_r,\beta)\subset K_{l-1}(O_r)$ by the surjection 
\eqref{eq:fundamental-surjection-for-general-alg-group}. 

Take a $\rho\in\frak{g}_{\beta}(\Bbb F)\sphat$ which defines group
homomorphisms
$$
 \chi_{\rho}
 =\rho\otimes\left[x_{l-1}\mapsto\tau(\varpi^{-(l-1)}x)\right]
 :Z(\mathcal{H}_{\beta})
  =\frak{g}(\Bbb F)_{\overline\beta}\times O_{l-1}
 \to\Bbb C^{\times}
$$
and
$$
 \widetilde{\chi}_{\rho}:
  \left(\Bbb G{\times}_{\frak{g}(\Bbb F)}\Bbb M\right)_{\beta}
  \xrightarrow{\spadesuit}Z(\mathcal{H}_{\beta})
  \xrightarrow{\chi_{\beta}}\Bbb C^{\times}.
$$
On the other hand we have a group homomorphism
$$
 \widetilde{\psi}_0:
 \Bbb G{\times}_{\frak{g}(\Bbb F)}\Bbb M\to\Bbb C^{\times}
$$
defined by 
$\widetilde{\psi}_0(X;S_{l-1},T_{l-1})
 =\tau\left(\varpi^{-l}B(\lambda(X)+\varpi T,\beta)\right)$, and 
$\widetilde{\psi}_0\cdot\widetilde{\chi}_{\beta}$ is trivial on the
 kernel of the surjection
 \eqref{eq:fundamental-surjection-for-general-alg-group}
 and induces a group homomorphism $\psi_{\beta,\rho}\in Y_{\beta}$
 defined in subsection
 \ref{subsec:description-of-main-result-for-odd-r}. 

\subsection[]{}
\label{subsec:schrodinger-representation}
Fix a $\rho\in\frak{g}_{\beta}(\Bbb F)\sphat$. 

Let us determine the 2-cocycle of the group extension
\begin{equation}
 0\to Z(\mathcal{H}_{\beta})\to\mathcal{H}_{\beta}
  \xrightarrow{\clubsuit}\Bbb V_{\beta}\to 0
\label{eq:central-extension-of-h-beta}
\end{equation}
where $\clubsuit:\mathcal{H}_{\beta}\to\Bbb V_{\beta}$
is defined by 
$(X,s)\mapsto\dot X\npmod{\frak{g}_{\beta}(\Bbb F)}$. 
Fix a 
$\Bbb F$-linear section $v\mapsto[v]$ of the exact sequence
$$
 0\to\frak{g}_{\beta}(\Bbb F)\to\frak{g}(\Bbb F)
  \to\Bbb V_{\beta}\to 0
$$
$\Bbb V$-vector spaces and define a section 
$l:\Bbb V_{\beta}\to\mathcal{H}_{\beta}$ of the group
  extension \eqref{eq:central-extension-of-h-beta} by 
$l(v)=([v],0)$. Then we have
$$
 l(u)l(v)l(u+v)^{-1}
 =(0,2^{-1}\varpi^{l-2}B([X,Y],\beta)\npmod{\frak{p}^{l-1}})
$$
for 
$u=\dot{\overline X}, v=\dot{\overline Y}
 \in\Bbb V_{\beta}$ so that the 2-cocycle of the group extension 
\eqref{eq:central-extension-of-h-beta} is
$$
 \left[(\dot{\overline X},\dot{\overline Y})\mapsto
       2^{-1}\varpi^{l-2}B([X,Y],\beta)\npmod{\frak{p}^{l-1}}\right]
 \in Z^2(\Bbb V_{\beta},O_{l-1}).
$$
Define a group operation on 
$\Bbb H_{\beta}=\Bbb V_{\beta}\times Z(\mathcal{H}_{\beta})$
by 
$$
 \left(\dot{\overline X},z\right)\cdot\left(\dot{\overline Y},w\right)
 =\left(\dot{\overline{X+Y}},
    z+w+2^{-1}\varpi^{l-2}B([X,Y],\beta)\npmod{\frak{p}^{l-1}}\right).
$$
Then $\Bbb H_{\beta}$ is isomorphic to $\mathcal{H}_{\beta}$ by 
$(v,(Y,s))\mapsto([v]+Y,s)$.

 Let $H_{\beta}$ be the Heisenberg group
of the symplectic $\Bbb F$-space $\Bbb V_{\beta}$, that is 
$H_{\beta}=\Bbb V_{\beta}\times\Bbb C^1$ with a group
operation
$$
 (u,s)\cdot(v,t)
 =\left(u+v,s\cdot t\cdot
            \widehat\tau(2^{-1}\langle u,v\rangle_{\beta})
        \right).
$$
Then we have a surjective group homomorphism 
$$
 \Bbb H_{\beta}=\Bbb V_{\beta}\times Z(\mathcal{H}_{\beta})\to H_{\beta}
 \qquad
 ((v,z)\mapsto(v,\chi_{\rho}(z)).
$$
Fix a polarization 
$\Bbb V_{\beta}=\Bbb W^{\prime}\oplus\Bbb W$ of the symplectic
$\Bbb F$-space $\Bbb V_{\beta}$. Let us denote by 
$L^2(\Bbb W^{\prime})$ 
the complex vector space of the complex-valued functions $f$ on
$\Bbb W^{\prime}$ with inner product 
$(f,f^{\prime})
 =\sum_{w\in\Bbb W^{\prime}}f(w)\overline{f^{\prime}(w)}$.
The Schr\"odinger representation $(\pi^{\beta},L^2(\Bbb W^{\prime}))$ 
of $H_{\beta}$ associated with the polarization is defined for 
$(v,s)\in H_{\beta}$ and $f\in L^2(\Bbb W^{\prime})$ by
$$
 \left(\pi^{\beta}(v,s)f\right)(w)
 =s\cdot\widehat\tau\left(
   2^{-1}\langle v_-,v_+\rangle_{\overline\beta}
   +\langle w,v_+\rangle_{\overline\beta}\right)\cdot f(w+v_-)
$$
where $v=v_-+v_+\in\Bbb V_{\beta}$ with 
$v_-\in\Bbb W^{\prime}, v_+\in\Bbb W^{\prime}$. 

Now an irreducible
representation $(\pi^{\beta,\rho},L^2(\Bbb W^{\prime}))$ of 
$\Bbb H_{\beta}$ is defined by 
$\pi^{\beta,\rho}(v,z)=\pi^{\beta}(v,\chi_{\rho}(z))$, and an
irreducible representation 
$(\widetilde\pi^{\beta,\rho},L^2(\Bbb W^{\prime}))$ of 
$\Bbb G{\times}_{\frak{g}(\Bbb F)}\Bbb M$ is defined by
$$
 \widetilde\pi^{\beta,\rho}:
 \Bbb G{\times}_{\frak{g}(\Bbb F)}\Bbb M
 \xrightarrow{\spadesuit}\mathcal{H}_{\beta}
 \,\tilde{\to}\,\Bbb H_{\beta}
 \xrightarrow{\pi^{\beta,\rho}}
 GL_{\Bbb C}(L^2(\Bbb W^{\prime})).
$$
Then $\widetilde{\psi}_0\cdot\widetilde{\pi}_{\beta,\rho}$ is trivial
on the kernel of 
$(\ast):\Bbb G{\times}_{\frak{g}(\Bbb F)}\Bbb M\to K_{l-1}(O_r)$ so
that it induces an irreducible representation 
$\pi^{\beta,\rho}$ of $K_{l-1}(O_r)$ on $L^2(\Bbb W^{\prime})$. 

\begin{prop}
\label{prop:another-expression-of-pi-beta-psi}
Take a $g=1+\varpi^{l-1}T\npmod{\frak{p}^r}\in K_{l-1}(O_r)$ 
with $T\in\frak{gl}_n(O)$. Then we have 
$T\npmod{\frak{p}^{l-1}}\in\frak{g}(O_{l-1})$ and
$$
 \pi^{\beta,\rho}(g)
  =\tau\left(\varpi^{-l}B(T,\beta)
            -2^{-1}\varpi^{-1}B(T^2,\beta)\right)\cdot
   \rho(Y)\cdot\pi^{\beta}(v,1)
$$
where $\overline T=[v]+Y\in\frak{g}(\Bbb F)$ with 
$v\in\Bbb V_{\beta}$ and 
$Y\in\frak{g}_{\beta}(\Bbb F)$. In particular
$\pi^{\beta,\rho}(h)$ is the homothety $\psi_{\beta,\rho}(h)$ for all 
$h\in Z(O_r,\beta)$.
\end{prop}
\begin{proof}
By the definition we have
\begin{align*}
 \pi^{\beta,\psi}(&l(X)(1+\varpi^lS\npmod{\frak{p}^r}))\\
 =&\psi_0(X,0)\cdot\widetilde\pi^{\beta,\rho}([v]+Y;S_{l-1},0)\\
 =&\psi\left(l(Y)(1+\varpi^lS\npmod{\frak{p}^r})\right)\cdot
   \tau\left(\varpi^{-l}B(\lambda(X)-\lambda(Y),\beta)\right)\cdot
   \pi^{\beta}(v,1)\\
 =&\tau\left(
     \varpi^{-(l-1)}B(S,\beta)+\varpi^{-l}B(\lambda(X),\beta)\right)
      \cdot\rho(Y)\cdot\pi^{\beta}(v,1)
\end{align*}
where $X=[v]+Y\in\frak{g}(\Bbb F)$ with $v\in\Bbb V_{\beta}$
and $Y\in\frak{g}_{\beta}(\Bbb F)$. 
Put $1+\varpi^{l-1}T\equiv l(X)(1+\varpi^lS)\npmod{\frak{p}^r}$ with 
$X\in\frak{g}(\Bbb F)$ and $S\in\frak{g}(O)$. Then we have
$$
 1+\varpi^{l-1}T\equiv 1+\varpi^{l-1}\lambda(X)
 \npmod{\frak{p}^l}
$$
so that we have $T\npmod{\frak p}=X\in\frak{g}(\Bbb F)$ and 
$$
 \varpi S\equiv 
 T-\lambda(\overline T)-2^{-1}\varpi^{l-1}\lambda(\overline T)^2
 \npmod{\frak{p}^l}.
$$
Then we have
\begin{align*}
 \pi^{\beta,\rho}(g)
 =&\pi^{\beta,\rho}\left(l(\overline T)(1+\varpi^lS)_r\right)\\
 =&\tau\left(\varpi^{-(l-1)}B(S,\beta)
             +\varpi^{-l}B(\lambda(\overline T),\beta)\right)\cdot
   \rho(Y)\cdot\pi^{\beta}(v,1)\\
 =&\tau\left(\varpi^{-l}B(T,\beta)-2^{-1}\varpi^{-1}B(T^2,\beta)
             \right)\cdot\rho(Y)\cdot\pi^{\beta}(v,1).
\end{align*}
\end{proof}

This proposition shows that the irreducible representation 
$(\pi^{\beta,\rho},L^2(\Bbb W^{\prime}))$ $K_{l-1}(O_r)$ 
is exactly the irreducible representation $\pi_{\beta,\rho}$ defined
in Proposition
\ref{prop:existence-of-fundamental-rep-associated-with-psi}. 

\subsection[]{}
\label{subsec:action-of-g-o-f-beta}
Fix a $\rho\in\frak{g}_{\beta}(\Bbb F)\sphat$. 
In this subsection we will study the conjugate action of 
$g_r=g\npmod{\frak{p}^r}\in G(O_r,\beta)$ on $K_{l-1}(O_r)$ and on 
$\pi^{\beta,\rho}$. For any $X\in\frak{g}(\Bbb F)$, we have
$$
 g_r^{-1}l(X)g_r
 =l\left(\text{\rm Ad}(\overline g)^{-1}X)
         +\varpi^l\nu(X,g)\npmod{\frak{p}^r}\right)
$$
with $\nu(X,g)\in\frak{g}(O)$ such that 
$$
 \text{\rm Ad}(g)^{-1}\lambda(X)-\lambda(\text{\rm Ad}(g)^{-1}X)
 =\varpi\cdot\nu(X,g).
$$
Then we have
\begin{align*}
 &g_r^{-1}l(X)(1+\varpi^l(S+T)_r)g_r\\
 =&l(\text{\rm Ad}(\overline g)^{-1}X)
   \left(1+\varpi^l(\text{\rm Ad}(g)^{-1}S+\text{\rm Ad}(g)^{-1}T
          +\nu(X,g))\right)\npmod{\frak{p}^r}
\end{align*}
and an action of $g_r\in G(O_r,\beta)$ on 
$(X;S_{l-1},T_{l-1})
 \in\Bbb G{\times}_{\frak{g}(\Bbb F)}\Bbb M$ is defined by 
\begin{equation}
 (X;S_{l-1},T_{l-1})^{g_r}
 =\left(\text{\rm Ad}(\overline g)^{-1}X;
         (\text{\rm Ad}(g)^{-1}S)_{l-1},
         (\text{\rm Ad}(g)^{-1}T+\nu(X,g))_{l-1}\right)
\label{eq:action-on-fibre-product}.
\end{equation}
The action \eqref{eq:action-on-fibre-product} is compatible with the
action  
$$
 (X,s)^{g_r}=(\text{\rm Ad}(\overline g)^{-1}X,s)
$$ 
of $g_r\in G(O_r,\beta)$ on $(X,s)\in\mathcal{H}_{\beta}$ via
the surjection
\eqref{eq:fund-surjection-on-hesenberg-like-group-for-general-alg-group}. 
If
we put $X=[v]+Y\in\frak{g}(\Bbb F)$ with $v\in\Bbb V_{\beta}$
and $Y\in\frak{g}_{\beta}(\Bbb F)$, then we have
$$
 \text{\rm Ad}(\overline g)^{-1}X
 =[v\sigma_{\overline g}]+\gamma(v,\overline g)
                        +\text{\rm Ad}(\overline g)^{-1}Y
$$
in the notations of subsection
\ref{subsec:def-of-schur-multiplier-associated-with-alg-group}. So 
$g_r\in G(O_r,\beta)$ acts on $(v,(Y,s))\in\Bbb H_{\beta}$ by 
$$
 (v,(Y,s))^{g_r}
 =\left(v\sigma_{\overline g},
   (\text{\rm Ad}(\overline g)^{-1}Y+\gamma(v,\overline g),s)\right).
$$
In particular $g_r\in G_l(O_r,\beta)^{(c)}$ acts on 
$(v,z)\in\Bbb H_{\beta}$ by 
$$
 (v,z)^{g_r}
 =(v\sigma_{\overline g},(\gamma(v,\overline g),0)\cdot z).
$$
There exists a group homomorphism 
$T:Sp(\Bbb V_{\beta})\to GL_{\Bbb C}(L^2(\Bbb W^{\prime}))$
  such that 
$$
 \pi^{\beta}(v\sigma,s)
 =T(\sigma)^{-1}\circ\pi^{\beta}(v,s)\circ T(\sigma)
$$
for all $\sigma\in Sp(\Bbb V_{\beta})$ and $(v,s)\in
H_{\beta}$ (see \cite[Th.2.4]{Gerardin1977}). 
Then we have
\begin{align*}
 \pi^{\beta,\rho}((v,z)^{g_r})
 =&\pi^{\beta,\rho}
    \left(v\sigma_{\overline g},(\gamma(v,\overline g),0)\cdot z\right)\\
 =&\pi^{\beta}
    \left(v\sigma_{\overline g},
           \rho(\gamma(v,\overline g))\cdot\chi_{\rho}(z)\right)\\
 =&\pi^{\beta}(v\sigma_{\overline g},
    \widehat\tau\left(\langle v,v_{\overline g}\rangle_{\beta}
                       \right)\cdot\chi_{\rho}(z))\\
 =&T(\sigma_{\overline g})^{-1}\circ
   \pi^{\beta}(v,
    \widehat\tau\left(\langle v,v_{\overline g}\rangle_{\beta}
                       \right)\cdot\chi_{\rho}(z))\circ
   T(\sigma_{\overline g})\\
 =&T(\sigma_{\overline g})^{-1}\circ
   \pi^{\beta}
    \left((v_{\overline g},1)^{-1}(v,\chi_{\rho}(z))(v_{\overline g},1)
           \right)\circ
   T(\sigma_{\overline g})\\
 =&T(\sigma_{\overline g})^{-1}\circ
   \pi^{\beta}(v_{\overline g},1)^{-1}\circ
   \pi^{\beta,\rho}(v,z)\circ
   \pi^{\beta}(v_{\overline g},1)\circ
   T(\sigma_{\overline g}).
\end{align*}
If we put
$$
 U(g_r)
 =\pi^{\beta}(v_{\overline g},1)\circ T(\sigma_{\overline g})
 \in GL_{\Bbb C}(L^2(\Bbb W^{\prime}))
$$
for $g_r\in G_l(O_r,\beta)^{(c)}$ then we have
$$
 U(g_r)\circ U(h_r)
 =c_{\beta,\rho}(\overline g,\overline h)\cdot U((gh)_r)
$$
for all $g_r,h_r\in G_l(O_r,\beta)^{(c)}$, in fact 
\begin{align*}
 U(g_r)\circ U(h_r)
 &=\pi^{\beta}(v_{\overline g},1)\circ T(\sigma_{\overline g})\circ 
   \pi^{\beta}(v_{\overline h},1)\circ T(\sigma_{\overline h})\\
 &=\pi^{\beta}(v_{\overline g},1)\circ 
   \pi^{\beta}(v_{\overline h}^{{\overline g}^{-1}},1)\circ
   T(\sigma_{\overline g})\circ T(\sigma_{\overline h})\\
 &=\pi^{\beta}\left(v_{\overline g}+v_{{\overline h}^{{\overline g}^{-1}}},
    \widehat\tau\left(
      2^{-1}\langle v_{\overline g},v_{{\overline h}^{{\overline g}^{-1}}}\rangle_{\beta}
                    \right)\right)\circ
   T(\sigma_{\overline{gh}})\\
 &=c_{\beta,\rho}(\overline g,\overline h)\cdot
   \pi^{\beta}(v_{\overline{gh}},1)\circ T(\sigma_{\overline{gh}}).
\end{align*}
On the other hand 
\begin{align*}
  \widetilde\psi_0\left((X;S_{l-1},T_{l-1})^{g_r^{-1}}
                       \right)
 &=\tau\left(\varpi^{-l}B(\lambda(X)+\varpi T,\text{\rm Ad}(g)\beta)
             \right)\\
 &=\tau\left(\varpi^{-l}B(\lambda(X)+\varpi T,\beta)
             \right)
\end{align*}
for all $g_r\in G(O_r,\beta)^{(c)}$. That is $\widetilde{\psi}_0$
is invariant under the conjugate action of $G_l(O_r,\beta)^{(c)}$. 
Hence we have
$$
 \pi^{\beta,\psi}(g_r^{-1}hg_r)
 =U(g_r)^{-1}\circ\pi^{\beta,\psi}(h)\circ
  U(g_r)
$$
for all $g_r\in G_l(O_r,\beta)^{(c)}$ and $h\in K_{l-1}(O_r)$. 

\section{Properties of $c_{\beta,\rho}$}
\label{sec:properties-of-c-beta-rho}
In this section we will present several properties of the Schur
multiplier 
$[c_{\beta,\rho}]\in H^2(G_{\beta}(\Bbb F)^{(c)},\Bbb C^{\times})$
defined in subsection
\ref{subsec:def-of-schur-multiplier-associated-with-alg-group}. 
We will keep the notations and conventions of section
\ref{sec:main-results}, while 
 only the structure of algebraic groups over finite field 
$\Bbb F$ is required to define and to discuss the Schur multiplier.

\subsection[]{}
\label{subsec:scalor-extension}
Let $\Bbb K/\Bbb F$ be a finite field extension. More specifically
take the unramified extension $K/F$ with the integer ring 
$O_K\subset K$ such that $\Bbb K=O_K/\varpi O_K$. Then 
$\frak{g}{\otimes}_OO_K$ is the Lie algebra of smooth $O_K$-group
scheme $G{\otimes}_OO_K$ and 
$$
 \frak{g}(\Bbb K)=(\frak{g}{\otimes}_OO_K)(\Bbb K)
                 =\frak{g}(\Bbb F){\otimes}_{\Bbb F}\Bbb K.
$$
So $B:\frak{g}(\Bbb K)\times\frak{g}(\Bbb K)\to\Bbb K$ is
non-degenerate. 

We will assume that the degree of the extension $(\Bbb K:\Bbb F)$ is
not divisible by the characteristic of $\Bbb F$, and put 
$T^{\prime}_{\Bbb K/\Bbb F}
 =(\Bbb K:\Bbb F)^{-1}T_{\Bbb K/\Bbb F}$. Then the additive character
$$
 \widetilde\tau:\Bbb K\xrightarrow{T^{\prime}_{\Bbb K/\Bbb F}}
                \Bbb F\xrightarrow{\widehat\tau}
                \Bbb C^{\times}
$$
is an extension of $\widehat\tau:\Bbb F\to\Bbb C^{\times}$. 

Fix a $\beta\in\frak{g}(O)$ such that 
$\frak{g}_{\beta}(\Bbb F)\lvertneqq\frak{g}(\Bbb F)$. Then for any 
$\rho\in\frak{g}_{\beta}(\Bbb F)$, the additive character
$$
 \widetilde\rho:(\frak{g}{\otimes}_OO_K)_{\beta}(\Bbb K)
                =(\frak{g}_{\beta})(\Bbb F){\otimes}_{\Bbb F}\Bbb K
                \xrightarrow{1\otimes T^{\prime}_{\Bbb K/\Bbb F}}
                \frak{g}_{\beta}(\Bbb F)\xrightarrow{\rho}
                \Bbb C^{\times}
$$
is an extension of $\rho:\frak{g}_{\beta}(\Bbb F)\to\Bbb C^{\times}$. 
Then we have 

\begin{prop}
\label{prop:scalor-extension-and-schur-multiplier-associated-with-alg-group}
The Schur multiplier 
$[c_{\beta,\rho}]\in H^2(G(\Bbb F)_{\beta}^{(c)},\Bbb C^{\times})$ is
the image of 
$[c_{\beta,\widetilde\rho}]
 \in H^2(G(\Bbb K)_{\beta}^{(c)},\Bbb C^{\times})$ under the
 restriction mapping 
$$
 \text{\rm Res}:H^2(G(\Bbb K)_{\beta}^{(c)},\Bbb C^{\times})
                \to
                H^2(G(\Bbb F)_{\beta}^{(c)},\Bbb C^{\times}).
$$
\end{prop}
\begin{proof}
There exists a $\Bbb F$-basis
$\{\alpha_{\lambda}\}_{\lambda\in\Lambda}$ of $\Bbb K$ such that 
$T_{\Bbb K/\Bbb F}(\alpha_{\lambda})=1$ for all $\lambda\in\Lambda$. 
Let $v\mapsto[v]$ be a $\Bbb F$-linear section of the exact sequence 
\eqref{eq:fundamental-exact-sequence-of-lie-alg-and-symplectic-space}. 
Its $\Bbb K$-linear extension gives a $\Bbb K$-linear section of the
exact sequence 
$$
 0\to\frak{g}_{\beta}(\Bbb F){\otimes}_{\Bbb F}\Bbb K
  \to\frak{g}(\Bbb F){\otimes}_{\Bbb F}\Bbb K
  \to\Bbb V_{\beta}{\otimes}_{\Bbb F}\Bbb K\to 0.
$$
Take a 
$g\in G(\Bbb F)_{\beta}^{(c)}\subset G(\Bbb K)_{\beta}^{(c)}$ and
$$
 v=\sum_{\lambda\in\Lambda}v_{\lambda}\otimes\alpha_{\lambda}
 \in\Bbb V_{\beta}{\otimes}_{\Bbb F}\Bbb K
 \qquad
 (v_{\lambda}\in\Bbb V_{\beta}). 
$$
Then we have
\begin{align*}
 \gamma(v,g)
 &=\text{\rm Ad}(g)^{-1}[v]-[v\sigma_g]\\
 &=\sum_{\lambda\in\Lambda}
    \left(\text{\rm Ad}(g)^{-1}[v_{\lambda}]-[v_{\lambda}\sigma_g]
           \right)\otimes\alpha_{\lambda}
  =\sum_{\lambda\in\Lambda}\gamma(v_{\lambda},g)\otimes\alpha_{\lambda},
\end{align*}
hence 
\begin{align*}
 \widetilde\rho\left(\gamma(v,g)\right)
 &=\prod_{\lambda\in\Lambda}\rho\left(
    \gamma((\Bbb K:\Bbb F)^{-1}v_{\lambda},g)\right)
  =\prod_{\lambda\in\Lambda}\widehat\tau\left(
    (\Bbb K:\Bbb F)^{-1}\langle
  v_{\lambda},v_g\rangle_{\beta}\right)\\
 &=\widetilde\tau\left(
    \sum_{\lambda\in\Lambda}\langle v_{\lambda},v_g\rangle_{\beta}
                        \alpha_{\lambda}\right)
  =\widetilde\tau\left(\langle v,v_g\rangle_{\beta}\right).
\end{align*}
So we have $c_{\beta,\rho}(g,h)=c_{\beta,\widetilde\rho}(g,h)$ for all 
$g,h\in G(\Bbb F)_{\beta}^{(c)}\subset G(\Bbb K)_{\beta}^{(c)}$.
\end{proof}

\subsection[]{}
\label{subsec:schur-multiplier-and-over-group}
Let us assume that there exists a closed smooth $O$-group subscheme 
$H\subset GL_n$ of which our $G$ is a closed $O$-group subscheme and
that the trace form
$$
 B:\frak{h}(\Bbb F)\times\frak{h}(\Bbb F)\to\Bbb F
$$ 
is non-degenerate where $\frak{h}$ is the Lie algebra of $H$. 
Then we have 
$\frak{h}(\Bbb F)=\frak{g}(\Bbb F)\oplus\frak{g}(\Bbb F)^{\perp}$ where 
$\frak{g}(\Bbb F)^{\perp}
 =\{X\in\frak{h}(\Bbb F)\mid\widetilde B(X,\frak{g}(\Bbb F))=0\}$ 
is the orthogonal complement of $\frak{g}(\Bbb F)$ in 
$\frak{h}(\Bbb F)$. 

Take a $\beta\in\frak{g}(O)$ such that 
$\frak{g}_{\beta}(\Bbb F)\lvertneqq\frak{g}(\Bbb F)$. Then 
$\beta\in\frak{h}(O)$ and 
$\frak{h}_{\beta}(\Bbb F)\lvertneqq\frak{h}(\Bbb F)$ where 
$\frak{h}_{\beta}=Z_{\frak h}(\beta)$ is the centralizer. We have
decompositions 
$\frak{h}_{\beta}(\Bbb F)
 =\frak{g}_{\beta}(\Bbb F)\oplus
  \left(\frak{g}(\Bbb F)^{\perp}\right)_{\beta}$ where 
$\left(\frak{g}(\Bbb F)^{\perp}\right)_{\beta}
 =\frak{h}_{\beta}(\Bbb F)\cap\frak{g}(\Bbb F)^{\perp}$, and
$$
 \widetilde{\Bbb V}_{\beta}
 =\frak{h}(\Bbb F)/\frak{h}_{\beta}(\Bbb F)
 =\Bbb V_{\beta}\oplus
  \left(\frak{g}(\Bbb F)^{\perp}/
   \left(\frak{g}(\Bbb F)^{\perp}\right)_{\beta}\right)
$$
is an orthogonal decomposition of symplectic spaces. 

Let $v\mapsto[v]$ be a $\Bbb F$-linear section of the exact sequence 
$$
 0\to\frak{h}_{\beta}(\Bbb F)\to\frak{h}(\Bbb F)
  \to\widetilde{\Bbb V}_{\beta}\to 0
$$
of $\Bbb F$-vector space such that $[\Bbb V_{\beta}]=\frak{g}(\Bbb F)$
and  
$[\frak{g}(\Bbb F)^{\perp}/\left(\frak{g}(\Bbb F)^{\perp}\right)_{\beta}]
 =\frak{g}(\Bbb F)^{\perp}$. 

Take $\rho\in\frak{g}_{\beta}(\Bbb F)\sphat$ and put
$$
 \widetilde\rho:\frak{h}_{\beta}
 =\frak{g}_{\beta}\oplus\left(\frak{g}^{\perp}\right)_{\beta}
 \xrightarrow{\text{\rm projection}}\frak{g}_{\beta}
 \xrightarrow{\rho}\Bbb C^{\times}.
$$
For any $g\in G_{\beta}(\Bbb F)\subset H_{\beta}(\Bbb F)$, there
exists uniquely a $v_g\in\Bbb V_{\beta}$ such that
$$
 \rho\left(\gamma_{\frak g}(v,g)\right)
 =\widehat\tau\left(\langle v,v_g\rangle_{\beta}\right)
$$
for all $v\in\Bbb V_{\beta}$. Then we have
$$
 \widetilde\rho\left(\gamma_{\frak{h}}(v,v_g)\right)
 =\widehat\tau\left(\langle v,v_g\rangle_{\beta}\right)
$$
for all $v\in\widetilde{\Bbb V}_{\beta}$. In fact if we put 
$v=v^{\prime}+v^{\prime\prime}$ with $v^{\prime}\in\Bbb V_{\beta}$ and 
$v^{\prime\prime}\in
 \frak{g}(\Bbb F)^{\perp}/\left(\frak{g}(\Bbb F)^{\perp}\right)_{\beta}$,
then we have 
$\gamma_{\frak{h}}(v,g)
 =\gamma_{\frak g}(v^{\prime},g)
 +\gamma_{\frak{h}}(v^{\prime\prime},g)$ with 
$\gamma_{\frak{h}}(v^{\prime\prime},g)\in
 \left(\frak{g}(\Bbb F)^{\perp}\right)_{\beta}$, since 
$$
 \text{\rm Ad}(g)\frak{g}(\Bbb F)^{\perp}=\frak{g}(\Bbb F)^{\perp},
 \qquad
 \text{\rm Ad}(g)\left(\frak{g}(\Bbb F)^{\perp}\right)_{\beta}
 =\left(\frak{g}(\Bbb F)^{\perp}\right)_{\beta}.
$$
Then we have
\begin{align*}
 \widetilde\rho\left(\gamma_{\frak{h}}(v,g)\right)
 &=\rho\left(\gamma_{\frak g}(v^{\prime},g)\right)
  =\widehat\tau\left(\langle v^{\prime},v_g\rangle_{\beta}\right)\\
 &=\widehat\tau\left(\langle v,v_g\rangle_{\beta}\right)
\end{align*}
because $\langle v^{\prime\prime},v_g\rangle_{\beta}=0$. Hence we have

\begin{prop}
\label{prop:regular-schur-multiplier-is-restriction-from-upper-group}
If $G(\Bbb F)_{\beta}^{(c)}\subset H(\Bbb F)_{\beta}^{(c)}$ then 
the Schur multiplier 
      $[c_{\beta,\rho}]
       \in H^2(G(\Bbb F)_{\beta}^{(c)},\Bbb C^{\times})$ is the image
       under the restriction mapping 
$$
 \text{\rm Res}:H^2(H(\Bbb F)_{\beta}^{(c)},\Bbb C^{\times})\to
                H^2(G(\Bbb F)_{\beta}^{(c)},\Bbb C^{\times})
$$
      of the Schur multiplier 
      $[c_{\beta,\widetilde\rho}]
       \in H^2(H(\Bbb F)_{\beta}^{(c)},\Bbb C^{\times})$.
\end{prop}

\subsection[]{}\label{subsec:schur-multiplier-and-jordan-decomposition}
Take a $\beta\in\frak{g}(O)$ such that 
$\frak{g}_{\beta}(\Bbb F)\lvertneqq\frak{g}(\Bbb F)$. 

Let $\overline\beta=\beta_s+\beta_n$ be the Jordan decomposition of 
$\overline\beta=\beta\npmod{\frak p}\in\frak{g}(\Bbb F)$ 
($\beta_s,\beta_n\in\frak{g}(\Bbb F)$ are repectively the semisimple
part and the nilpotent part of $\overline\beta$). Put
$$
 L=Z_{G{\otimes}_O\Bbb F}(\beta_s),
 \quad
 \frak{l}=\text{\rm Lie}(L)(\Bbb F)
         =Z_{\frak{g}(\Bbb F)}(\beta_s).
$$
Let us assume the following assumption;

\begin{assu}\label{assume:b-over-l-is-non-degenerate}
$B:\frak{l}\times\frak{l}\to\Bbb F$ is non-degenerate.
\end{assu}

Then we have $\frak{g}(\Bbb F)=\frak{l}\oplus\frak{l}^{\perp}$ where
$$
 \frak{l}^{\perp}
 =\{X\in\frak{g}(\Bbb F)\mid B(X,\frak{l})=0\}
$$
is an $\text{\rm Ad}(L(\Bbb F))$-submodule of $\frak{g}(\Bbb F)$. Hence there
exists a $\Bbb F$-vector subspace 
$\frak{l}^{\perp}\subset V\subset\frak{g}(\Bbb F)$ such that 
$\frak{g}(\Bbb F)=V\oplus\frak{g}_{\beta}(\Bbb F)$. We can fix a 
$\Bbb F$-linear section $v\mapsto[v]$ 
on $\Bbb V_{\beta}$ of the exact sequence 
$$
 0\to\frak{g}_{\beta}(\Bbb F)\to\frak{g}(\Bbb F)\to\Bbb V_{\beta}
  \to 0
$$
such that $[v]\in V$ for all $v\in\Bbb V_{\beta}$. Then we have

\begin{prop}\label{prop:schur-multiplier-and-jordan-decomposition}
Fix a $\rho\in\frak{g}_{\beta}(\Bbb F)\sphat$. 
Under the assumption \ref{assume:b-over-l-is-non-degenerate}, we have
\begin{enumerate}
\item $X_g=[v_g]\in\frak{l}$ for any $g\in G_{\beta}(\Bbb F)^{(c)}$,
\item for any extension $\widetilde\rho\in\frak{l}\sphat$ of $\rho$,
  we have
$$
 c_{\beta,\rho}(g,h)
 =\widetilde\rho\left(2^{-1}\text{\rm Ad}(g)X_h\right)\cdot
  \widetilde\rho\left(2^{-1}X_{gh}\right)\cdot
  \widetilde\rho\left(2^{-1}X_g\right)
$$
  for all $g,h\in G_{\beta}(\Bbb F)^{(c)}$.
\end{enumerate}
\end{prop}
\begin{proof}
1) Take any $X\in\frak{l}^{\perp}$ and put 
$v=\dot X\in
 \Bbb V_{\beta}=\frak{g}(\Bbb F)/\frak{g}_{\beta}(\Bbb F)$. Then we
 have
$$
 \gamma(v,g)=\text{\rm Ad}(g)^{-1}[v]-[v\sigma_g]=0
$$
because $[v]=X$ and 
$\text{\rm Ad}(g)^{-1}X\in\frak{l}^{\perp}\subset V$ for any 
$g\in G_{\beta}(\Bbb F)^{(c)}$. Then we have
$$
 \widehat\tau(B([X,X_g],\overline\beta)
 =\widehat\tau(\langle v,v_g\rangle_{\beta})
 =\rho(\gamma(v,g))=1
$$
and hence
$$
 B(X,[X_g,\overline\beta])=B([X,X_g],\overline\beta)=0
$$
for all $X\in\frak{l}^{\perp}$. This means that 
$[X_g,\overline\beta]\in(\frak{l}^{\perp})^{\perp}=\frak{l}$. Since
$[\beta_s,\beta]=0$ and $\frak{l}=Z_{\frak{g}(\Bbb F)}(\beta_s)$, we
have
$$
 [[\beta_s,X_g],\overline\beta]=[\beta_s,[X_g,\overline\beta]]=0
$$
that is 
$[\beta_s,X_g]\in\frak{g}_{\beta}(\Bbb F)\subset\frak{l}$. Finally 
$$
 B(Y,[\beta_s,X_g])=B([Y,\beta_s],X_g)=0
$$
for all $Y\in\frak{l}$, so we have $[\beta_s,X_g]=0$, that is 
$X_g\in\frak{l}$.

2) Take $g,h\in G_{\beta}(\Bbb F)^{(c)}$. The relation 
\ref{eq:multiplication-formula-of-v-epsilon} gives 
$v_g\sigma_g=-v_{g^{-1}}$ and 
$$
 [v_h\sigma_{g^{-1}}]=X_{gh}-X_g\in\frak{l}.
$$
Then we have
\begin{align*}
 c_{\beta,\rho}(g,h)
 &=\widehat\tau\left(2^{-1}\langle v_h,v_{g^{-1}}\rangle_{\beta}\right)
  =\rho\left(2^{-1}\gamma(v_h,g^{-1})\right)\\
 &\widetilde\rho\left(2^{-1}\{\text{\rm Ad}(g)X_h-(X_{gh}-X_g)\}
                       \right).
\end{align*}
\end{proof}

\begin{rem}\label{remark:v-g-is-trivial-if-trace-form-is-non-deg-on-g-beta}
Fix a $\beta\in\frak{g}(O)$. If 
$B:\frak{g}_{\beta}(\Bbb F)\times\frak{g}_{\beta}(\Bbb F)\to\Bbb F$ is
non-degenerate, then we have 
$\frak{g}(\Bbb F)
 =\frak{g}_{\beta}(\Bbb F)^{\perp}\oplus\frak{g}_{\beta}(\Bbb F)$
 where
$$
 \frak{g}_{\beta}(\Bbb F)^{\perp}
 =\{X\in\frak{g}(\Bbb X)\mid B(X,\frak{g}_{\beta}(\Bbb F))=0\}
$$
is a $\text{\rm Ad}(G_{\beta}(\Bbb F))$-invariant 
$\Bbb F$-subspace. So if we choose a $\Bbb F$-linear section 
$v\mapsto[v]$ on 
$\Bbb V_{\beta}=\frak{g}(\Bbb F)/\frak{g}_{\beta}(\Bbb F)$ of the
canonical exact sequence of $\Bbb F$-vector space
$$
 0\to\frak{g}_{\beta}(\Bbb F)\to\frak{g}(\Bbb F)\to\Bbb V\to 0
$$
so that $[v]\in\frak{g}_{\beta}(\Bbb F)^{\perp}$ for all $v\in\Bbb V_{\beta}$,
then we have $v_g=0$ for all $g\in G_{\beta}(\Bbb F)$. 
In this case have $c_{\beta,\rho}(g,h)=1$ for all 
$g,h\in G_{\beta}(\Bbb F)^{(c)}$.
\end{rem}


\subsection[]{}\label{subsec:smoothly-regular-in-reductive-group}
In this subsection, we will consider the relation between the
regularity of $\overline\beta\in\frak{g}(\Bbb F)$ and the triviality
of the Schur multiplier 
$[c_{\beta,\rho}]\in H^2(G_{\beta}(\Bbb F)^{(c)},\Bbb C^{\times})$. 

Put $K=F$ or $K=\Bbb F$. 
Let us suppose that $G{\otimes}_O K$ is
connected reductive algebraic $ K$-group and the characteristic of
$ K$ is not bad with respect to $G{\otimes}_O K$. 
The list of the bad primes is 
\begin{center}
\begin{tabular}{c|c|c|c|c|c|c}
type of $G{\otimes}_O K$   
   &$A_r$      &$B_r, D_r$&$C_r$&$E_6, E_7, F_4$&$E_8$&$G_2$\\
\hline
bad prime&$\emptyset$&$2$       &$2$  &$2,3$        &$2,3,5$&$2,3$
\end{tabular}
\end{center}
(see \cite[p.178, I-4.3]{Borel1968}). 

Take a $\beta\in\frak{g}(O)$ and let $\overline\beta=\beta_s+\beta_n$
be the Jordan decomposition of $\overline\beta\in\frak{g}( K)$
into the semi-simple part $\beta_s\in\frak{g}( K)$ and the
nilpotent part $\beta_n\in\frak{g}( K)$
($\overline\beta\in\frak{g}(K)$ is the image of $\beta\in\frak{g}(O)$
under the canonical mapping $\frak{g}(O)\to\frak{g}(K)$). 
The identity component 
$L=Z_{G{\otimes}_O K}(\beta_s)^o$ of the centralizer of $\beta_s$ in 
$G{\otimes}_O K$ is a reductive group over $ K$ and there
exists a maximal torus $T$ of $G{\otimes}_O K$ such that
$$
 \beta_s\in\text{\rm Lie}(T)( K)
$$
(see \cite[Prop.13.19]{Borel1991} and its proof). 
Then $T\subset L$ and 
$\text{\rm rank}(L)=\text{\rm rank}(G)$. Put 
$\frak{l}=\text{\rm Lie}(L)$, then 
$\frak{l}=Z_{\frak{g}{\otimes}_O K}(\beta_s)$. 
So $\overline\beta\in\frak{g}( K)$ is smoothly regular with
respect to $G{\otimes}_O K$ if and only if 
$\beta_n\in\frak{l}( K)$ is smoothly regular with respect to $L$. 

Now fix a system of
positive roots $\Phi^+$ in the root system $\Phi(T,L)$ of
$L$ with respect to $T$ such that 
$$
 \beta_n=\sum_{\alpha\in\Phi^+}c_{\alpha}\cdot X_{\alpha}
$$
where $X_{\alpha}$ is a root vector of the root $\alpha$. Then 
the result of \cite[p.228,III-3.5]{Borel1968} implies 

\begin{prop}
\label{prop:regularity-in-terms-of-coefficient-of-root-vector}
$\overline\beta\in\frak{g}( K)$ is smoothly regular with respect
to $G$ over $K$ if and only if 
$c_{\alpha}\neq 0$ for all simple $\alpha\in\Phi^+$. 
\end{prop}

\begin{rem}\label{remark:smoothly-regular-in-gl-n}
Let us consider the case of $GL_n$ which is a connected smooth
reductive $O$-group scheme. For $\beta\in\frak{gl}_n(O)$, the
following statements are equivalent;
\begin{enumerate}
\item $\beta\in\frak{gl}_n(O)$ is smoothly regular with
      respect to $GL_n$ over $\Bbb F$,
\item $\overline\beta\in M_n(\Bbb F)$ is 
      $GL_n(\overline{\Bbb F})$-conjugate to 
$$
 J_{n_1}(\alpha_1)\boxplus\cdots\boxplus J_{n_r}(\alpha_r)=
 \begin{bmatrix}
  J_{n_1}(\alpha_1)&      &                 \\
                   &\ddots&                 \\
                   &      &J_{n_r}(\alpha_r)
 \end{bmatrix},
$$
      where $\alpha_1,\cdots,\alpha_r$ are distinct elements of the
      algebraic closure $\overline{\Bbb F}$ of $\Bbb F$ and
$$
 J_m(\alpha)=\begin{bmatrix}
              \alpha&   1  &      &      \\
                    &\alpha&\ddots&      \\
                    &      &\ddots&   1  \\
                    &      &      &\alpha\\
             \end{bmatrix}
$$
      is a Jordan block of size $m$,
\item the characteristic polynomial 
      $\chi_{\overline\beta}(t)
       =\det(t1_n-\overline\beta)\in\Bbb F[t]$ is the
      minimal polynomial of $\overline\beta\in M_n(\Bbb F)$,
\item $\{X\in M_n(\Bbb F)\mid X\overline\beta=\overline\beta X\}
       =\Bbb F[\overline\beta]$,
\item $\{X\in M_n(O_l)\mid X\beta\equiv\beta X\pmod{\frak{p}^l}\}
       =O_l[\beta_l]$ with
       $\beta_l=\beta\pmod{\frak{p}^l}$ for all $l>0$,
\item $\{X\in M_n(O)\mid X\beta=\beta X\}=O[\beta]$.
\end{enumerate}
In this case $\beta\in\frak{gl}_n(O)$ is smoothly regular with respect
to $GL_n$ over $F$ and the centralizer $GL_{n,\beta}$ is commutative
and smooth over $O$.
\end{rem}

The remark above and Proposition
\ref{prop:schur-multiplier-and-jordan-decomposition} combined with the
results of \cite[Examples 4.6.2--4]{Takase2016}, we have

\begin{prop}\label{prop:c-beta-rho-is-trivial-for-gl-n}
Take a $\beta\in\frak{gl}_n(O)$ such that 
\begin{enumerate}
\item $\beta$ is smoothly regular with respect to $GL_n$ over $\Bbb F$,
\item the multiplicities of the roots
      of the characteristic polynomial $\chi_{\overline\beta}(t)$ of 
      $\overline\beta\in M_n(\Bbb F)$ are at most $4$. 
\end{enumerate}
Then the Schur multiplier 
$[c_{\beta,\rho}]\in H^2(GL_{n,\beta}(\Bbb F)^{(c)},\Bbb C^{\times})$ is
trivial for all $\rho\in\frak{gl}_{n,\beta}(\Bbb F)\sphat$ provided that
the characteristic of $\Bbb F$ is big enough.
\end{prop}

The author does not hesitate to present the following 

\begin{conj}
\label{conjecture:schur-multiplier-is-trivial-for-all-rho-if-regular}
The Schur multiplier 
$[c_{\beta,\rho}]\in H^2(GL_{n,\beta}(\Bbb F)^{(c)},\Bbb C^{\times})$
is trivial for all $\rho:\frak{gl}_{n,\beta}(\Bbb F)\sphat$, 
if $\beta\in\frak{gl}_n(O)$ is smoothly regular with respect to 
 $GL_n$ over $\Bbb F$ and the characteristic of $\Bbb F$ is big enough. 
\end{conj}

In general if $\beta\in\frak{g}(O)\subset\frak{gl}_n(O)$ is smoothly
regular with respect to $GL_n$ over $\Bbb F$, then $\beta$ is also
smoothly regular with respect to $GL_n$ over $F$, hence 
$GL_{n,\beta}$ is smooth commutative $O$-group scheme. So 
$G_{\beta}\subset GL_{n,\beta}$ is commutative, and so we have
$$
 G_{\beta}(\Bbb F)^{(c)}=G_{\beta}(\Bbb F)
 \subset GL_{n,\beta}(\Bbb F)=GL_{n,\beta}(\Bbb F)^{(c)}.
$$
Then Proposition
\ref{prop:regular-schur-multiplier-is-restriction-from-upper-group}
says that the Schur multiplier 
$[c_{\beta,\rho}]\in H^2(G_{\beta}(\Bbb F)^{(c)},\Bbb C^{\times})$
with $\rho\in\frak{g}_{\beta}(\Bbb F)\sphat$ is the image under the
restriction mapping
$$
 \text{\rm Res}:H^2(GL_{n,\beta}(\Bbb F)^{(c)},\Bbb C^{\times})
                \to
                H^2(G_{\beta}(\Bbb F)^{(c)},\Bbb C^{\times})
$$
of the Schur multiplier 
$[c_{\beta,\widetilde\rho}]\in 
 H^2(GL_{n,\beta}(\Bbb F)^{(c)},\Bbb C^{\times})$. 

If $G\subset GL_n$ is a group symplectic similitude, 
a general orthogonal group with
respect to a quadratic form of odd variables or an unitary group 
(see section \ref{sec:example}), then 
Proposition
\ref{prop:regularity-in-terms-of-coefficient-of-root-vector} implies
the following; 
$\beta\in\frak{g}(O)\in\frak{gl}_n(O)$ is smoothly regular with respect to 
$G$ over $\Bbb F$ if and only if $\beta$ is smoothly regular with 
respect to $GL_n$ over $\Bbb F$. So Conjecture 
\ref{conjecture:schur-multiplier-is-trivial-for-all-rho-if-regular}
implies the following conjectural proposition

\begin{prop}
\label{prop:conjectural-proposition-on-triviality-of-schur-multiplier}
If $G\subset GL_n$ is a group of symplectic similitudes, 
a general orthogonal group with
respect to a quadratic form of odd variables or an unitary group, then 
the Schur multiplier 
$[c_{\beta,\rho}]\in H^2(G(\Bbb F)_{\beta}^{(c)},\Bbb C^{\times})$ is
trivial for all $\rho\in\frak{g}_{\beta}(\Bbb F)\sphat$ 
if $\beta\in\frak{g}(O)$ is regular with respect to $G$ over $\Bbb F$ 
and the characteristic of $\Bbb F$ is big enough. 
\end{prop}

These arguments provide good reasons for the
conjectural statement of subsection 
\ref{subsec:conjectures-on-triviality-of-c-beta-rho}. 

\section{Examples}\label{sec:example}

\subsection[]{}\label{subsec:general-linear-group}
$G=GL_n$ is a connected smooth reductive $O$-group scheme which
satisfies the conditions I), II) and III) of the subsection 
\ref{subsec:fundamental-setting}. 

\begin{prop}\label{prop:congruent-beta-for-gl-n}
If 
$\beta\equiv\lambda\cdot 1_n+\beta_0\npmod{\frak{p}^{l^{\prime}}}$ 
for $\beta, \beta_0\in\frak{gl}_n(O)$ with $\lambda\in O$, then there
  exists a group homomorphism 
$\mu:O_r^{\times}\to\Bbb C^{\times}$ such that 
$$
 \pi\mapsto(\mu\circ\det)\otimes\pi
$$
gives a bijection of $\text{\rm Irr}(G(O_r)\mid\psi_{\beta_0})$ onto 
$\text{\rm Irr}(G(O_r)\mid\psi_{\beta})$.
\end{prop}
\begin{proof}
Let $\mu:O_r^{\times}\to\Bbb C^{\times}$ be a group homomorphism such
that $\mu((1+\varpi^l)_rx)=\tau(\varpi^{-l^{\prime}}\lambda x)$ for all 
$x\in O$. Then we have 
$$
 \psi_{\beta}(h)=\mu\circ\det(h)\cdot\psi_{\beta_0}(h)
$$
for all 
$h=1_n+\varpi^lX\npmod{\frak{p}^r}\in K_l(O_r)$ ($X\in
M_n(O)$), because
$$
 \det(1_n+\varpi^lX)\equiv 1+\varpi^l\text{\rm tr}(X)
 \npmod{\frak{p}^r}.
$$
\end{proof}

A similar argument shows 

\begin{prop}\label{prop:central-beta-for-gl-n}
If $\beta\in\frak{gl}_n(O)$ is central modulo $\frak{p}$, then there
exists a group homomorphism $\mu:O_r^{\times}\to\Bbb C^{\times}$ such
that, for any 
$\pi\in\text{\rm Irr}(G(O_r)\mid\psi_{\beta})$, there exists a 
$\sigma\in\text{\rm Irr}(G(O_{r-1}))$ such that 
$\pi=(\mu\circ\det)\otimes(\sigma\circ\text{\rm proj})$ where 
$\text{\rm proj}:G(O_r)\to G(O_{r-1})$ is the canonical surjection.
\end{prop}
\begin{proof}
Put $\beta\equiv\lambda\cdot 1_n\npmod{\frak p}$ with 
$\lambda\in O$. 
Let $\mu:O_r^{\times}\to\Bbb C^{\times}$ be a group homomorphism such
that $\mu((1+\varpi^l)_rx)=\tau(\varpi^{-l^{\prime}}\lambda x)$ for all 
$x\in O$. 
Then, for any $\pi\in\text{\rm Irr}(G(O_r)\mid\psi_{\beta})$, the
representation $(\mu\circ\det)^{-1}\otimes\pi$ of $G(O_r)$ is trivial
on $K_{r-1}(O_r)$, in other word, it factors through $G(O_{r-1})$.
\end{proof}

If $\beta\in\frak{gl}_n(O)$ is
smoothly regular with respect to $GL_n$ over $\Bbb F$, then 
$G_{\beta}$ is a smooth commutative $O$-group scheme (see Remark 
\ref{remark:smoothly-regular-in-gl-n}). If further 
$\overline\beta\in\frak{gl}_n(\Bbb F)$ is semisimple, 
then $c_{\beta,\rho}(g,h)=1$ for all $g,h\in G_{\beta}(\Bbb F)$ and 
$\rho\in\frak{gl}_{n,\beta}(\Bbb F)\sphat$ 
(see Remark
\ref{remark:v-g-is-trivial-if-trace-form-is-non-deg-on-g-beta}). 
In this case Theorem \ref{th:main-result} gives

\begin{prop}\label{prop:regular-semisimple-case-for-gl-n}
There exists a bijection 
$\theta\mapsto
 \text{\rm Ind}_{G(O_r,\beta)}^{G(O_r)}\sigma_{\beta,\theta}$ of the
 set 
$$
 \left\{\theta\in G_{\beta}(O_r)\sphat\;\;\;
         \text{\rm s.t. $\theta=\psi_{\beta}$ on 
               $G_{\beta}(O_r)\cap K_l(O_r)$}\right\}
$$
onto $\text{\rm Irr}(G(O_r)\mid\psi_{\beta})$.
\end{prop}

There are $\frac 1n\sum_{d|n}\mu\left(\frak nd\right)\cdot q^d$ 
irreducible polynomials in $\Bbb F[t]$ of degree $n$ ($\mu$ is the
M\"obius function). Take a polynomial 
$$
 p(t)=t^n+a_1t^{n-1}+\cdots+a_{n-1}t+a_n\in O[t]
$$ 
such that $p(t)\npmod{\frak p}\in\Bbb F[t]$ is
irreducible. Then 
$$
 [p(t)]=\left\{\beta\in\frak{gl}_n(O)\mid 
           \chi_{\beta}(t)\equiv p(t)\npmod{\frak p}\right\}
$$
is stable under the adjoint action of $GL_n(O)$. 
Take a $\beta\in [p(t)]$. Then  
$K=F[\beta]$ is an unramified field extension of $F$ of degree $n$,
which is the splitting field of $p(t)$ over $F$, and
$O_K=O_F[\beta]$ is the integer ring of $K$ with the maximal ideal 
$\frak{p}_K=\varpi O_K$. If we identify 
$G_{\beta}(O_r)$ with $\left(O_K/\frak{p}_K^r\right)^{\times}$, then we have
$$
 G_{\beta}(O_r)\cap K_{l^{\prime}}(O_r)
 =(1+\frak{p}_K^{l^{\prime}})/(1+\frak{p}_K^r)
 \hookrightarrow
 \left(O_K/\frak{p}_K^r\right)^{\times}
$$
and 
$\psi_{\beta}\left((1+\varpi^lx)_r\right)
 =T_{K/F}(\varpi^{-l^{\prime}}x\beta)$ ($x\in O_K$). 
So Proposition \ref{prop:regular-semisimple-case-for-gl-n} gives

\begin{prop}\label{prop:shintani-gerardin-parametrization}
There exists a bijection 
$\theta\mapsto
 \text{\rm Ind}_{G(O_r,\beta)}^{G(O_r)}\sigma_{\beta,\theta}$ of the
 set 
$$
 \left\{\begin{array}{l}
        \theta:O_K^{\times}\to\left(O_K/\frak{p}_K^r\right)^{\times}
               \to\Bbb C^{\times}:\text{\rm group homomorphism}\\
        \text{\rm such that 
  $\theta(1+\varpi^lx)=T_{K/F}(\varpi^{-l^{\prime}}\beta x)
   \;\;
   \forall x\in O_K$}
         \end{array}\right\}
$$
onto $\text{\rm Irr}(G(O_r)\mid\psi_{\beta})$.
\end{prop}

This kind of parametrization of 
$\text{\rm Irr}(G(O_r)\mid\psi_{\beta})$ is given first by Shintani 
\cite{Shintani1968} and then G\'erardin \cite{Gerardin1972}.

\subsection[]{}\label{subsec:group-of-symplectic-similitude}
Let $G=GSp_{2n}$ be the $O$-group scheme such that
$$
 GSp_{2n}(L)=\{g\in GL_{2n}(L)\mid gJ_n\,^t\!g=\nu(g)\cdot J_n
                \;\text{\rm with $\nu(g)\in L^{\times}$}\}
$$
 ($J_n=\begin{bmatrix}
        0&1_n\\
       -1_n&0
       \end{bmatrix}$) 
for all commutative $O$-algebra $L$. Then $G$ is a connected smooth
reductive $O$-group scheme. The Lie algebra $\frak{gsp}_{2n}$ of $G$ 
is an affine $O$-subscheme of $\frak{gl}_{2n}$ such that
$$
 \frak{gsp}_{2n}(L)=\{X\in\frak{gl}_{2n}(L)\mid
                        XJ_n+J_n\,^t\!X=\nu(X)\cdot J_n\;
                      \text{\rm with $\nu(X)\in L$}\}
$$
for all commutative $O$-algebra $L$. 
Assume that the characteristic of $\Bbb F$ does not divide $n$. Then
$G$ satisfies the conditions I), II) and III) of the subsection 
\ref{subsec:fundamental-setting}. The same arguments as Propositions 
\ref{prop:congruent-beta-for-gl-n} and \ref{prop:central-beta-for-gl-n} 
show

\begin{prop}\label{prop:congruent-beta-for-gsp-2n}
If 
$\beta\equiv\lambda\cdot 1_n+\beta_0\npmod{\frak{p}^{l^{\prime}}}$ 
for $\beta, \beta_0\in\frak{gsp}_{2n}(O)$ with $\lambda\in O$, then there
  exists a group homomorphism 
$\mu:O_r^{\times}\to\Bbb C^{\times}$ such that 
$$
 \pi\mapsto(\mu\circ\det)\otimes\pi
$$
gives a bijection of $\text{\rm Irr}(G(O_r)\mid\psi_{\beta_0})$ onto 
$\text{\rm Irr}(G(O_r)\mid\psi_{\beta})$.
\end{prop}

and

\begin{prop}\label{prop:central-beta-for-gsp-2n}
If $\beta\in\frak{gsp}_{2n}(O)$ is central modulo $\frak{p}$, then there
exists a group homomorphism $\mu:O_r^{\times}\to\Bbb C^{\times}$ such
that, for any 
$\pi\in\text{\rm Irr}(G(O_r)\mid\psi_{\beta})$, there exists a 
$\sigma\in\text{\rm Irr}(G(O_{r-1}))$ such that 
$\pi=(\mu\circ\det)\otimes(\sigma\circ\text{\rm proj})$ where 
$\text{\rm proj}:G(O_r)\to G(O_{r-1})$ is the canonical surjection.
\end{prop}

If $\beta\in\frak{gsp}_{2n}(O)$ is
smoothly regular with respect to $G$ over $F$ and over $\Bbb F$, then 
$G_{\beta}$ is a smooth commutative $O$-group scheme. If further 
$\overline\beta\in\frak{gsp}_{2n}(\Bbb F)$ is semisimple, 
then $c_{\beta,\rho}(g,h)=1$ for all 
$g,h\in G_{\beta}(\Bbb F)^{(c)}=G_{\beta}(\Bbb F)$ and 
$\rho\in\frak{gsp}_{2n,\beta}(\Bbb F)\sphat$
(see Remark
\ref{remark:v-g-is-trivial-if-trace-form-is-non-deg-on-g-beta}). 
In this case Theorem \ref{th:main-result} gives

\begin{prop}\label{prop:regular-semisimple-case-for-gsp-2n}
There exists a bijection 
$\theta\mapsto
 \text{\rm Ind}_{G(O_r,\beta)}^{G(O_r)}\sigma_{\beta,\theta}$ of the
 set 
$$
 \left\{\theta\in G_{\beta}(O_r)\sphat\;\;\;
         \text{\rm s.t. $\theta=\psi_{\beta}$ on 
               $G_{\beta}(O_r)\cap K_l(O_r)$}\right\}
$$
onto $\text{\rm Irr}(G(O_r)\mid\psi_{\beta})$.
\end{prop}

Let us consider a $\beta\in\frak{gsp}_{2n}(O)\subset\frak{gl}_{2n}(O)$
such that the characteristic polynomial $\chi_{\beta}(t)\in O[t]$ is
irreducible modulo $\frak{p}$. In this case $F[\beta]$ is unramified
field extension of $F$ of degree $2n$ such that $O[\beta]$ is the
integer ring of $F[\beta]$ so that $\beta\in O[\beta]^{\times}$. 

Now fix an unramified field extension $K$ of $F$ of degree $2n$. Let 
$O_K$ be the integer ring of $K$ with the maximal ideal 
$\frak{p}_K=\varpi O_K$. The finite field $\Bbb F=O/\frak{p}$ is
identified with a subfield of $\Bbb K=O_K/\frak{p}_K$. Then $K/F$ is a
Galois extension whose Galois group $\text{\rm Gal}(K/F)$ is
isomorphic to $\text{\rm Gal}(\Bbb K/\Bbb F)$ by the mapping which
sends $\sigma\in\text{\rm Gal}(K/F)$ to 
$\overline\sigma\in\text{\rm Gal}(\Bbb K/\Bbb F)$ where 
$(x\npmod{\frak{p}_K})^{\overline\sigma}=x^{\sigma}\npmod{\frak{p}_K}$. 

Let $\tau\in\text{\rm Gal}(K/F)$ be the unique element of order $2$. 
Fix an $\varepsilon\in O_K^{\times}$ such that 
$\varepsilon^{\tau}+\varepsilon=0$. Then the $F$-vector space $K$ is a
symplectic $F$-space with respect to the non-degenerate alternating
form 
$$
 D_{\varepsilon}(x,y)=T_{K/F}(\varepsilon xy^{\tau})
 \qquad
 (x,y\in K)
$$
on $K$ with a polarization $K=K_-\oplus K_+$ where 
$$
 K_{\pm}=\{x\in K\mid x^{\tau}=\pm x\},
$$
and $F\subset K_+\subset K$ is a subfield such that $(K_+:F)=n$. 
Let $\{v_1,\cdots,v_n\}$ be a $O$-basis of $O_{K_+}=O_K\cap K_+$ the
integer ring of $K_+$. Since $K/F$ is unramified, we have a $O$-basis
of $O_{K_+}$ $\{v_1^{\ast},\cdots,v_n^{\ast}\}$ such that 
$T_{K_+/F}(v_iv_j^{\ast})=\delta_{ij}$. Put 
$u_i=\varepsilon^{-1}v_i^{\ast}\subset K_-$. Then 
$\{u_1,\cdots,u_n,v_1,\cdots,v_n\}$ 
is a symplectic $F$-basis of $K=K_-\oplus K_+$ and an $O$-basis of 
$O_K=\varepsilon^{-1}O_{K_+}\oplus O_{K_+}$. 

Now our $O$-group scheme $G=GSp_{2n}$ and its Lie algebra 
$\frak{gsp}_{2n}$ are defined by
$$
 GSp_{2n}(L)=\left\{g\in GL_L(O_K{\otimes}L)\biggm|
                    \begin{array}{l}
                    D_{\varepsilon}(gx,gy)=\nu(g)D_{\varepsilon}(x,y)\\
                    \text{\rm for $\forall x,y\in O_K{\otimes}_OL$ with 
                              $\nu(g)\in L^{\times}$}
                    \end{array}\right\}
$$
and by
$$
 \frak{gsp}_{2n}(L)
  =\left\{X\in\text{\rm End}_L(O_K{\otimes}L)\biggm|
           \begin{array}{l}
            D_{\varepsilon}(Xx,y)+D_{\varepsilon}(x,Xy)
                       =\nu(X)D_{\varepsilon}(x,y)\\
            \text{\rm for $\forall x,y\in O_K{\otimes}_OL$ 
                               with $\nu(X)\in L$}
           \end{array}\right\}
$$
for all commutative $O$-algebra $L$. 

Take a $\beta\in K_-\cap O_K$ such that $O_K=O[\beta]$. Identify 
$\beta\in K$ with the element $x\mapsto x\beta$ of 
$\frak{gsp}_{2n}(O)\subset\frak{gl}_{2n}(O)$. 
Then the characteristic polynomial $\chi_{\beta}(t)\in O[t]$ of 
$\beta\in\frak{gl}_{2n}(O)$ is irreducible modulo $\frak{p}$. We have 
$$
 G_{\beta}(O_r)=G(O_r)\cap\left(O_K/\frak{p}_K^r\right)^{\times}
               =\{\gamma\npmod{\frak{p}_K^r}\mid
                    \gamma\in U_{K/F}\}
$$
where
$$
 U_{K/F}=\{\gamma\in O_K^{\times}\mid 
            \gamma\cdot\gamma^{\tau}\in O^{\times}\}
$$
is a subgroup of $O_K^{\times}$. In this case we have 
$\psi_{\beta}(h)=\tau(\varpi^{-l^{\prime}}T_{K/F}(\beta x))$ for all 
$$
 h=1+\varpi^lx\npmod{\frak{p}_K^r}\in K_l(O_r)\cap G_{\beta}(O_r)
 \subset\left(O_K/\frak{p}_K^r\right)^{\times}.
$$
Then Proposition \ref{prop:regular-semisimple-case-for-gsp-2n} gives 

\begin{prop}\label{prop:shintani-gerardin-type-parametrization-for-gsp-2n}
There exists a bijection 
$\theta\mapsto
 \text{\rm Ind}_{G(O_r,\beta)}^{G(O_r)}\sigma_{\beta,\theta}$ of the
 set 
$$
 \left\{\begin{array}{l}
        \theta:U_{K/F}\to\left(O_K/\frak{p}_K^r\right)^{\times}
               \to\Bbb C^{\times}:\text{\rm group homomorphism}\\
        \text{\rm such that 
  $\theta(\gamma)
   =\tau(\varpi^{-l^{\prime}}T_{K/F}(\varpi^{-l^{\prime}}\beta x))
   \;\;
   \forall\gamma=1+\varpi^lx\in U_{K/F}$}
         \end{array}\right\}
$$
onto $\text{\rm Irr}(G(O_r)\mid\psi_{\beta})$.
\end{prop}

\subsection[]{}\label{subsec:general-orthogonal-group}
Take a $S\in M_n(O)$ such that $^tS=S$ and $\det S\in O^{\times}$. 
Let $G=GO(S)$ be the $O$-group scheme such tha
$$
 G(L)=\{g\in GL_n(L)\mid gS^t\!g=\nu(g)\cdot S\;
          \text{\rm with $\nu(g)\in L^{\times}$}\}
$$
for all commutative $O$-algebra $L$. Then $G$ is a connected smooth
reductive $O$-group scheme. The Lie algebra $\frak{g}=\frak{go}(S)$ of
$G$ is an affine $O$-subscheme of $\frak{gl}_n$ such that
$$
 \frak{g}(L)=\{X\in\frak{gl}_n(L)\mid 
                 XS+S^t\!X=\nu(X)\cdot S\;
                 \text{\rm with $\nu(X)\in L$}\}
$$
for all commutative $O$-algebra $L$. 
Assume that the characteristic of $\Bbb F$ does not divide $n$. Then
$G$ satisfies the conditions I), II) and III) of the subsection 
\ref{subsec:fundamental-setting}. 
The same arguments as Propositions 
\ref{prop:congruent-beta-for-gl-n} and \ref{prop:central-beta-for-gl-n} 
show

\begin{prop}\label{prop:congruent-beta-for-go(s)}
If 
$\beta\equiv\lambda\cdot 1_n+\beta_0\npmod{\frak{p}^{l^{\prime}}}$ 
for $\beta, \beta_0\in\frak{g}(O)$ with $\lambda\in O$, then there
  exists a group homomorphism 
$\mu:O_r^{\times}\to\Bbb C^{\times}$ such that 
$$
 \pi\mapsto(\mu\circ\det)\otimes\pi
$$
gives a bijection of $\text{\rm Irr}(G(O_r)\mid\psi_{\beta_0})$ onto 
$\text{\rm Irr}(G(O_r)\mid\psi_{\beta})$.
\end{prop}

and

\begin{prop}\label{prop:central-beta-for-go(s)}
If $\beta\in\frak{g}(O)$ is central modulo $\frak{p}$, then there
exists a group homomorphism $\mu:O_r^{\times}\to\Bbb C^{\times}$ such
that, for any 
$\pi\in\text{\rm Irr}(G(O_r)\mid\psi_{\beta})$, there exists a 
$\sigma\in\text{\rm Irr}(G(O_{r-1}))$ such that 
$\pi=(\mu\circ\det)\otimes(\sigma\circ\text{\rm proj})$ where 
$\text{\rm proj}:G(O_r)\to G(O_{r-1})$ is the canonical surjection.
\end{prop}

If $\beta_s\in\frak{g}(L)$ ($L=F$ or $L=\Bbb F$) is semisimple, then the
centralizer $Z_{G{\otimes}_OL}(\beta_s)$ is connected and its center is
also connected. Hence if $\beta\in\frak{g}(O)$ is smoothly regular
over $F$ and over $\Bbb F$, then $G_{\beta}$ is a smooth commutative
$O$-group scheme. If further $\overline\beta\in\frak{g}(\Bbb F)$ is
semisimple, then $c_{\beta,\rho}(g,h)=1$ for all 
$g,h\in G_{\beta}(\Bbb F)^{(c)}=G_{\beta}(\Bbb F)$ and all 
$\rho\in\frak{g}_{\beta}(\Bbb F)\sphat$ (see Remark 
\ref{remark:v-g-is-trivial-if-trace-form-is-non-deg-on-g-beta}). 
In this case Theorem \ref{th:main-result} gives

\begin{prop}\label{prop:regular-semisimple-case-for-go(s)}
There exists a bijection 
$\theta\mapsto
 \text{\rm Ind}_{G(O_r,\beta)}^{G(O_r)}\sigma_{\beta,\theta}$ of the
 set 
$$
 \left\{\theta\in G_{\beta}(O_r)\sphat\;\;\;
         \text{\rm s.t. $\theta=\psi_{\beta}$ on 
               $G_{\beta}(O_r)\cap K_l(O_r)$}\right\}
$$
onto $\text{\rm Irr}(G(O_r)\mid\psi_{\beta})$.
\end{prop}

Suppose $n=2m$ is even. Let $K$ be an unramified field extension of
$F$ of degree $n$ with the integer ring $O_K$. Then 
$\frak{p}_K=\varpi O_K$ is the maximal ideal of $O_K$, and 
$\Bbb F=O/\frak{p}$ is canonically identified with a subfield of 
$\Bbb K=O_K/\frak{p}_K$. Let $\tau\in\text{\rm Gal}(K/F)$ be the
unique element of order $2$. Fix an $\varepsilon\in O_K^{\times}$ such
that $\varepsilon^{\tau}=\varepsilon$ and put
$$
 S_{\varepsilon}(x,y)=T_{K/F}(\varepsilon xy^{\tau})
 \qquad
 (x,y\in K).
$$
Then $S$ is a regular $F$-quadratic form on $K$. 
Fix a $O$-basis $\{u_1,\cdots,u_n\}$  of $O_K$ and put 
$B=\left(u_i^{\sigma_j}\right)_{1\leq i,j\leq n}\in M_n(O_K)$ where 
$\text{\rm Gal}(K/F)=\{\sigma_1,\cdots,\sigma_n\}$. Then we have 
$$
 \left(S_{\varepsilon}(u_i,u_j)\right)_{1\leq i,j\leq n}
 =B\begin{bmatrix}
     \varepsilon^{\sigma_1}&      &                      \\
                           &\ddots&                      \\
                           &      &\varepsilon^{\sigma_n}
   \end{bmatrix}\,^tB^{\tau}
$$ 
so that the discriminant of the quadratic form $S_{\varepsilon}$ is
$$
 \det(S_{\varepsilon}(u_i,u_j))_{1\leq i,j\leq n}=
 (-1)^m(\det B)^2N_{K/F}(\varepsilon). 
$$
Note that $(\det B)^{\sigma}=-\det B$ for a generator $\sigma$ of 
$\text{\rm Gal}(K/F)$. Since $K/F$ is unramified extension,
its discriminant $D(K/F)=(\det B^2)$ is trivial, in other word 
$\det(S_{\varepsilon}(u_i,u_j))_{1\leq i,j\leq n}\in O^{\times}$. So
the $O$-group 
scheme $G=GO(S_{\varepsilon})$ and its Lie algebra
$\frak{g}=\frak{go}(S_{\varepsilon})$ is 
defined by 
$$
 G(L)=\left\{g\in GL_L(O_K{\otimes}L)\biggm|
                    \begin{array}{l}
                    S_{\varepsilon}(xg,yg)=\nu(g)S_{\varepsilon}(x,y)\\
                    \text{\rm for $\forall x,y\in O_K{\otimes}_OL$ with 
                              $\nu(g)\in L^{\times}$}
                    \end{array}\right\}
$$
and by 
$$
 \frak{g}(L)
  =\left\{X\in\text{\rm End}_L(O_K{\otimes}L)\biggm|
           \begin{array}{l}
            S_{\varepsilon}(xX,y)+S_{\varepsilon}(x,yX)
                       =\nu(X)S_{\varepsilon}(x,y)\\
            \text{\rm for $\forall x,y\in O_K{\otimes}_OL$ 
                               with $\nu(X)\in L$}
           \end{array}\right\}
$$
for all commutative $O$-algebra $L$. Note that $\text{\rm End}_F(K)$
acts on $K$ from the right side.

Take a $\beta\in O_K$ such that $O_K=O[\beta]$ and $\beta^{\tau}=-\beta$. 
Identify $\beta\in K$ with the element $x\mapsto x\beta$ of 
$\frak{g}(O)\subset\text{\rm End}_O(O_K)=\frak{gl}_{2n}(O)$. 
Then the characteristic polynomial $\chi_{\beta}(t)\in O[t]$ of 
$\beta\in\frak{gl}_{2n}(O)$ is irreducible modulo $\frak{p}$. We have 
$$
 G_{\beta}(O_r)=G(O_r)\cap\left(O_K/\frak{p}_K^r\right)^{\times}
               =\{\gamma\nnpmod{\frak{p}_K^r}\mid
                    \gamma\in U_{K/F}\}
$$
where
$$
 U_{K/F}=\{\gamma\in O_K^{\times}\mid 
            \gamma\cdot\gamma^{\tau}\in O^{\times}\}
$$
is a subgroup of $O_K^{\times}$. In this case we have 
$\psi_{\beta}(h)=\tau(\varpi^{-l^{\prime}}T_{K/F}(\beta x))$ for all 
$$
 h=1+\varpi^lx\npmod{\frak{p}_K^r}\in K_l(O_r)\cap G_{\beta}(O_r)
 \subset\left(O_K/\frak{p}_K^r\right)^{\times}.
$$
Then Proposition \ref{prop:regular-semisimple-case-for-go(s)} gives 

\begin{prop}
\label{prop:shintani-gerardin-type-parametrization-for-go(s)-of-even-n}
There exists a bijection 
$\theta\mapsto
 \text{\rm Ind}_{G(O_r,\beta)}^{G(O_r)}\sigma_{\beta,\theta}$ of the
 set 
$$
 \left\{\begin{array}{l}
        \theta:U_{K/F}\to\left(O_K/\frak{p}_K^r\right)^{\times}
               \to\Bbb C^{\times}:\text{\rm group homomorphism}\\
        \text{\rm such that 
  $\theta(\gamma)
   =\tau(\varpi^{-l^{\prime}}T_{K/F}(\varpi^{-l^{\prime}}\beta x))
   \;\;
   \forall\gamma=1+\varpi^lx\in U_{K/F}$}
         \end{array}\right\}
$$
onto $\text{\rm Irr}(G(O_r)\mid\psi_{\beta})$.
\end{prop}

Let us consider the case of $n=2m+1$ being odd. In this case 
$\det\beta=0$ for all $\beta\in\frak{g}(O)$ so that the characteristic
polynomial $\chi_{\beta}(t)\in O[t]$ of $\beta$ is of the form 
$\chi_{\beta}(t)=t\cdot p(t)$ with $p(t)\in O[t]$. We will give
examples in which $p(t)$ is irreducible modulo $\frak{p}$. 

Let us use the notations used in the case of even $n$ and fix an 
$\eta\in O^{\times}$. Then $F$-quadratic form on $V=K\times F$ is
defined by 
$$
 S_{\varepsilon,\eta}((x,s),(y,t))
 =S_{\varepsilon}(x,y)+\eta\cdot st.
$$
Then the $O$-group scheme $G=GO(S_{\varepsilon,\eta})$ and its Lie
algebra $\frak{g}=\frak{go}(S_{\varepsilon,\eta})$ is defined by
$$
 G(L)=\left\{g\in GL_L((O_K\times O){\otimes}L)\biggm|
                    \begin{array}{l}
                    S_{\varepsilon,\eta}(ug,vg)
                         =\nu(g)S_{\varepsilon,\eta}(u,v)
                    \;\;\text{\rm for}\\
                    \text{\rm $\forall u,v\in(O_K\times O){\otimes}_OL$ with 
                              $\nu(g)\in L^{\times}$}
                    \end{array}\right\}
$$
and
\begin{align*}
 &\frak{g}(L)\\
 &=\left\{X\in\text{\rm End}_L((O_K\times O){\otimes}L)\biggm|
           \begin{array}{l}
            S_{\varepsilon,\eta}(uX,v)+S_{\varepsilon,\eta}(u,vX)
                       =\nu(X)S_{\varepsilon,\eta}(u,v)\\
            \text{\rm for $\forall u,v\in(O_K\times O){\otimes}_OL$ 
                               with $\nu(X)\in L$}
           \end{array}\right\}
\end{align*}
for all commutative $O$-algebra $L$. An element 
$X\in\text{\rm End}_L((O_K\times O){\otimes}_OL)$ is denoted by
$$
 X=\begin{bmatrix}
    A&B\\
    C&D
   \end{bmatrix}\;\text{\rm with}\;
 \begin{cases}
  A\in\text{\rm End}_L(O_K{\otimes}_OL),
    &B\in\text{\rm Hom}_L(O_K{\otimes}_OL,L),\\
  C\in\text{\rm Hom}_L(L,O_K{\otimes}_OL),
    &D\in\text{\rm End}_L(L)=L.
 \end{cases}
$$
Take a $\beta_0\in O_K^{\times}$ such that $O_K=O[\beta_0]$ and 
$\beta_0^{\tau}=-\beta_0$. Then the element 
$(x,s)\mapsto(x\beta_0,0)$ of $\frak{g}(O)$ is denoted by
$$
 \beta=\begin{bmatrix}
        \beta_0&0\\
           0   &0
       \end{bmatrix}\in\frak{g}(O)
      \subset\text{\rm End}_O(O_K\times O),
$$
where $\beta_0\in O_K^{\times}$ is identified with 
$x\mapsto x\beta_0$ an element of $\text{\rm End}_O(O_K)$.  Now the
characteristic polynomial of 
$\beta\in\frak{g}(O)\subset\text{\rm End}_O(O_K\times O)$ is 
$\chi_{\beta}(t)=t\cdot\chi_{\beta_0}(t)$ where 
$\chi_{\beta_0}(t)\in O[t]$ is the characteristic polynomial of 
$\beta_0\in\text{\rm End}_O(O_K)$ which is irreducible modulo
$\frak{p}$. In particular $\beta\npmod{\frak{p}}\in\frak{g}(\Bbb F)$
is semisimple. We have
\begin{align*}
 G_{\beta}(O_r)
 &=G(O_r)\cap O_r[\beta\npmod{\frak{p}^r}]^{\times}\\
 &=\left\{\begin{bmatrix}
           \gamma\npmod{\frak{p}_K^r}&0\\
           0&a\npmod{\frak{p}^r}
          \end{bmatrix}\biggm|
           \begin{array}{l}
            \gamma\in U_{K/F},\;a\in O^{\times}\\
            \gamma\gamma^{\tau}\equiv a^2\npmod{\frak{p}^r}
           \end{array}
          \right\}.
\end{align*}
In this case 
$\psi_{\beta}(h)=\tau\left(\varpi^{-l^{\prime}}T_{K/F}(\beta_0x)\right)$
for all
$$
 h=\begin{bmatrix}
    1+\varpi^lx\npmod{\frak{p}_K^r}&0\\
    0&1+\varpi^ly\npmod{\frak{p}^r}
   \end{bmatrix}
 \in K_l(O_r)\cap G_{\beta}(O_r).
$$
The Proposition \ref{prop:regular-semisimple-case-for-go(s)} gives

\begin{prop}
\label{prop:shintani-gerardin-type-parametrization-for-go(s)-of-odd-n}
There exists a bijection 
$\theta\mapsto
 \text{\rm Ind}_{G(O_r,\beta)}^{G(O_r)}\sigma_{\beta,\theta}$ of the
 set 
$$
 \left\{\begin{array}{l}
        \theta:U_{K/F}{\times}_{\frak{p}^r}O^{\times}
          \to\left(O_K/\frak{p}_K^r\right)^{\times}O_r^{\times}
               \to\Bbb C^{\times}\\
    \hphantom{\theta}:\text{\rm group homomorphism such that}\\
    \hphantom{theta:}\theta(\gamma,a)
   =\tau(\varpi^{-l^{\prime}}T_{K/F}(\varpi^{-l^{\prime}}\beta x))\\
    \hphantom{theta:}
      \text{\rm for $\forall(\gamma,a)=(1+\varpi^lx,1+\varpi^ly)
           \in U_{K/F}{\times}_{\frak{p}^r}O^{\times}$}
         \end{array}\right\}
$$
onto $\text{\rm Irr}(G(O_r)\mid\psi_{\beta})$. Here
$$
 U_{K/F}{\times}_{\frak{p}^r}O^{\times}
 =\{(\gamma,a)\in U_{K/F}\times O^{\times}\mid
     \gamma\gamma^{\tau}\equiv a^2\npmod{\frak{p}^r}\}
$$
is a subgroup of $U_{K/F}\times O^{\times}$.
\end{prop}

\subsection[]{}\label{sebsec:unitary-group}
Let $K/F$ be the unramified quadratic field extension and put 
$\text{\rm Gal}(K/F)=\langle\tau\rangle$. Let us denote by $O_K$ the
integer ring of $K$ and $\frak{p}_K=\varpi O_K$ the maximal ideal of
$O_K$. Fix a $S\in M_n(O_K)$ such that $^tS^{\tau}=S$ and 
$\det S\in O_K^{\times}$. Then the unitary group $G=U(S)$ associated
with the Hermitian form $S$ is the $O$-group scheme defined by
$$
 G(A)=\{g\in GL_n(O_K{\otimes}_O A)\mid gS\,^t\!g^{\tau\otimes 1}=S\}
$$
for all commutative $O$-algebra $A$. Its Lie algebra
$\frak{g}=\frak{u}(S)$ is an affine $O$-scheme defined by 
$$
 \frak{g}(A)=\{X\in\frak{gl}_n(O_K{\otimes}_O A)\mid 
                XS+S\,^tX^{\tau\otimes 1}=0\}.
$$
We can take an $\varepsilon\in O^{\times}$ such that
$O_K=O[\sqrt\varepsilon]$. If there exists an $\eta\in A$ such that
$\varepsilon 1_A=\eta^2$, then we can identify $O_K{\otimes}_O A$ with 
$A\oplus A$ such a way that $(x,y)^{\tau\otimes 1}=(y,x)$. Then 
we have identifications $G(A)=GL_n(A)$ and
$\frak{g}(A)=\frak{gl}_n(A)$. If we put 
$$
 K=\left\{\begin{bmatrix}
            x&\varepsilon y\\
            y&x
           \end{bmatrix}\biggm| x,y\in F\right\},
 \quad
 O_K=\left\{\begin{bmatrix}
            x&\varepsilon y\\
            y&x
           \end{bmatrix}\biggm| x,y\in O\right\}
$$
then
$$
 M_n(O_K)=\{X\in M_{2n}(O_F)\mid I_n(\varepsilon)X=XI_n(\varepsilon)\}
$$
where
$$
 I_n(\varepsilon)=\begin{bmatrix}
                   I(\varepsilon)&      &              \\
                                 &\ddots&              \\
                                 &      &I(\varepsilon)
                  \end{bmatrix}\in M_{2n}(O_F)
 \;\text{\rm with $I(\varepsilon)=\begin{bmatrix}
                                   0&\varepsilon\\
                                   1&0
                                  \end{bmatrix}$}.
$$
Then we have 
$$
 G(A)=\{g\in GL_{2n}(A)\mid I_n(\varepsilon)g=gI_n(\varepsilon), 
                            gS^t\!g^{\tau}=S\}
$$
and
$$
 \frak{g}(A)=\{X\in\frak{gl}_{2n}(A)\mid 
                I_n(\varepsilon)X=XI_n(\varepsilon),\;
                XS+S^t\!X^{\tau}=0\}
$$
for all commutative $O$-algebra $A$. In this way $G$ is a closed
$O$-group subscheme of $GL_{2n}$. Then $G$ satisfies the conditions
I), II) and III) of the subsection \ref{subsec:fundamental-setting}. 
The center of $\frak{g}(O)$ is
$$
 Z(\frak{g}(O))=\{\lambda\cdot 1_n\in GL_n(O_K)\mid
                   \lambda\in O_K\;
                   \text{\rm s.t. $\lambda^{\tau}+\lambda=0$}\}.
$$
The similar arguments as Propositions 
\ref{prop:congruent-beta-for-gl-n} and \ref{prop:central-beta-for-gl-n} 
show

\begin{prop}\label{prop:congruent-beta-for-u(s)}
If 
$\beta\equiv\lambda\cdot 1_n+\beta_0\npmod{\frak{p}^{l^{\prime}}}$ 
for $\beta, \beta_0\in\frak{g}(O)$ with $\lambda\in O_K$, then there
  exists a group homomorphism 
$\mu:O_{K,r}^{\times}\to\Bbb C^{\times}$ such that 
$$
 \pi\mapsto(\mu\circ\det)\otimes\pi
$$
gives a bijection of $\text{\rm Irr}(G(O_r)\mid\psi_{\beta_0})$ onto 
$\text{\rm Irr}(G(O_r)\mid\psi_{\beta})$.
\end{prop}

and

\begin{prop}\label{prop:central-beta-for-u(s)}
If $\beta\in\frak{g}(O)$ is central modulo $\frak{p}$, then there
exists a group homomorphism $\mu:O_{K,r}^{\times}\to\Bbb C^{\times}$ such
that, for any 
$\pi\in\text{\rm Irr}(G(O_r)\mid\psi_{\beta})$, there exists a 
$\sigma\in\text{\rm Irr}(G(O_{r-1}))$ such that 
$\pi=(\mu\circ\det)\otimes(\sigma\circ\text{\rm proj})$ where 
$\text{\rm proj}:G(O_r)\to G(O_{r-1})$ is the canonical surjection.
\end{prop}

If $\beta_s\in\frak{g}(L)$ ($L=F$ or $L=\Bbb F$) is semisimple, then the
centralizer $Z_{G{\otimes}_OL}(\beta_s)$ is connected and its center is
also connected. Hence if $\beta\in\frak{g}(O)$ is smoothly regular
over $F$ and over $\Bbb F$, then $G_{\beta}$ is a smooth commutative
$O$-group scheme. If further $\overline\beta\in\frak{g}(\Bbb F)$ is
semisimple, then $c_{\beta,\rho}(g,h)=1$ for all 
$g,h\in G_{\beta}(\Bbb F)^{(c)}=G_{\beta}(\Bbb F)$ and all 
$\rho\in\frak{g}_{\beta}(\Bbb F)\sphat$ (see Remark 
\ref{remark:v-g-is-trivial-if-trace-form-is-non-deg-on-g-beta}). 
In this case Theorem \ref{th:main-result} gives

\begin{prop}\label{prop:regular-semisimple-case-for-u(s)}
There exists a bijection 
$\theta\mapsto
 \text{\rm Ind}_{G(O_r,\beta)}^{G(O_r)}\sigma_{\beta,\theta}$ of the
 set 
$$
 \left\{\theta\in G_{\beta}(O_r)\sphat\;\;\;
         \text{\rm s.t. $\theta=\psi_{\beta}$ on 
               $G_{\beta}(O_r)\cap K_l(O_r)$}\right\}
$$
onto $\text{\rm Irr}(G(O_r)\mid\psi_{\beta})$.
\end{prop}

Assume that $n$ is odd. Let $L/F$ be the unramified field extension of
degree $2n$ so that $K$ is a subfield of $L$. There exists a 
$\tau\in\text{\rm Gal}(L/F)$ of order $2$ and 
$\text{\rm Gal}(K/F)=\langle\tau|_K\rangle$. Fix an 
$\varepsilon\in O_L^{\times}$ such that
$\varepsilon^{\tau}=\varepsilon$. Then 
$$
 S_{\varepsilon}(x,y)=T_{L/K}(\varepsilon xy^{\tau})
 \quad
 (x,y\in L)
$$
is a non-degenerate Hermitian form over $K$ on $L$. Let 
$\{u_1,\cdots,u_n\}$ be a $O_K$-basis of $O_L$. Then we have
$$
 (S_{\varepsilon}(u_i,u_j))_{1\leq i,j\leq n}
 =B\begin{bmatrix}
    \varepsilon^{\sigma_1}&      &                      \\
                          &\ddots&                      \\
                          &      &\varepsilon^{\sigma_n}
   \end{bmatrix}^t\!B^{\tau}
$$
where $B=(u_i^{\sigma_k})_{1\leq i,k\leq n}\in M_n(O_L)$ with 
$\text{\rm Gal}(L/K)=\{\sigma_1,\cdots,\sigma_n\}$. So the
discriminant of the Hermitian form $S_{\varepsilon}$ is 
$$
 \det(S_{\varepsilon}(u_i,u_j))_{1\leq i,j\leq n}
 =N_{L/K}(\varepsilon)\det B(\det B)^{\tau}\in O_K^{\times}
$$
because $L/K$ is unramified. Then the $O$-group scheme 
$G=U(S_{\varepsilon})$ and its Lie algebra
$\frak{g}=\frak{u}(S_{\varepsilon})$ are defined by
$$
 G(A)=\{g\in GL_A(O_L{\otimes}_OA)\mid 
          S_{\varepsilon}(gx,gy)=S_{\varepsilon}(x,y)\;
          \text{\rm for $\forall x,y\in L$}\}
$$
and
$$
 \frak{g}(A)=\{X\in\text{\rm End}_A(O_L{\otimes}_OA)\mid
                S_{\varepsilon}(Xx,y)+S_{\varepsilon}(x,Xy)=0\;
                \text{\rm for $\forall x,y\in L$}\}
$$
for all commutative $O$-algebra $A$. 

Take a $\beta\in O_L$ such that $\beta^{\tau}+\beta=0$ and 
$O_L=O[\beta]$. Identify $\beta\in L$ with the element 
$x\mapsto x\beta$ of 
$\frak{g}(O)\subset\text{\rm End}_O(O_L)=\frak{gl}_{2n}(O)$. Then the
characteristic polynomial $\chi_{\beta}(t)\in O[t]$ is irreducible
modulo $\frak{p}$. We have
$$
 G_{\beta}(O_r)=G(O_r)\cap\left(O_L/\frak{p}_L^r\right)^{\times}
  =\{\gamma\nnpmod{\frak{p}_L^r}\mid\gamma\in U_{L/F}\}
$$
where
$$
 U_{L/F}
 =\{\gamma\in O_L^{\times}\mid\gamma\gamma^{\tau}\in O^{\times}\}.
$$
In this case 
$\psi_{\beta}(h)=\tau(\varpi^{-l^{\prime}}T_{L/F}(\beta x))$ for all
$$
 h=1+\varpi^lx\npmod{\frak{p}_L^r}\in K_l(O_r)\cap G_{\beta}(O_r)
  \subset\left(O_L/\frak{p}_L^r\right)^{\times}.
$$
Then Proposition \ref{prop:regular-semisimple-case-for-u(s)} gives 

\begin{prop}\label{prop:shintani-gerardin-type-parametrization-for-u(s)}
There exists a bijection 
$\theta\mapsto
 \text{\rm Ind}_{G(O_r,\beta)}^{G(O_r)}\sigma_{\beta,\theta}$ of the
 set 
$$
 \left\{\begin{array}{l}
        \theta:U_{L/F}\to\left(O_L/\frak{p}_L^r\right)^{\times}
               \to\Bbb C^{\times}:\text{\rm group homomorphism}\\
        \text{\rm such that 
  $\theta(\gamma)
   =\tau(\varpi^{-l^{\prime}}T_{L/F}(\varpi^{-l^{\prime}}\beta x))
   \;\;
   \forall\gamma=1+\varpi^lx\in U_{K/F}$}
         \end{array}\right\}
$$
onto $\text{\rm Irr}(G(O_r)\mid\psi_{\beta})$.
\end{prop}


Sendai 980-0845, Japan\\
Miyagi University of Education\\
Department of Mathematics
\end{document}